\documentclass[a4paper]{article}

\usepackage{amsthm,amsfonts,amssymb,amsmath,amsxtra,latexsym}
\usepackage[all]{xy}
\usepackage{xr-hyper}
\usepackage[colorlinks=true]{hyperref}

\usepackage{mathrsfs}

\theoremstyle{plain}
\newtheorem{thm}{Theorem}[section]

\newtheorem{cor}[thm]{Corollary}
\newtheorem{prop}[thm]{Proposition}
\newtheorem{conj}[thm]{Conjecture}
\theoremstyle{definition}
\newtheorem{defn}[thm]{Definition}

\newtheorem{remark}[thm]{Remark}

    \newcommand{\FH}{{\mathbf{H}}} \newcommand{\FG}{{\mathbf{G}}}
    \newcommand{\FT}{{\mathbf{T}}}
    
    \newcommand{\FP}{{\mathbf{P}}}
    \renewcommand{\FT}{{\mathbf{T}}}\newcommand{\FM}{{\mathbf{M}}}
    \newcommand{\FN}{{\mathbf{N}}}\newcommand{\FA}{{\mathbf{A}}}
    \newcommand{\FZ}{{\mathbf{Z}}}
    \newcommand{\FS}{{\mathbf{S}}}
    \newcommand{\FL}{{\mathbf{L}}}\newcommand{\FX}{{\mathbf{X}}}

    \renewcommand{\1}{{\mathbbold{1}}}

    \newcommand{\sC}{{\mathscr{C}}}

    \newcommand{\BA}{{\mathbb {A}}} 
    \newcommand{\BC}{{\mathbb {C}}} \newcommand{\BD}{{\mathbb {D}}}
     
    \newcommand{\BG}{{\mathbb {G}}} \newcommand{\BH}{{\mathbb {H}}}

     \newcommand{\BN}{{\mathbb {N}}}
     
     \newcommand{\BR}{{\mathbb {R}}}

    \newcommand{\CC}{{\mathcal {C}}} \newcommand{\CD}{{\mathcal {D}}}

     \newcommand{\CN}{{\mathcal {N}}}
    \newcommand{\CO}{{\mathcal {O}}} 
     
    \newcommand{\CS}{{\mathcal {S}}}

    \newcommand{\gl}{{\mathfrak{gl}}}
     
    \newcommand{\fc}{{\mathfrak{c}}} 
     
    \newcommand{\fg}{{\mathfrak{g}}} \newcommand{\fh}{{\mathfrak{h}}}
     
     \newcommand{\fl}{{\mathfrak{l}}}
    \newcommand{\fm}{{\mathfrak{m}}} 
     
     \newcommand{\fr}{{\mathfrak{r}}}
    \newcommand{\fs}{{\mathfrak{s}}} \newcommand{\ft}{{\mathfrak{t}}}

    \newcommand{\B}{{\mathrm{B}}}
    
    \newcommand{\ad}{{\mathrm{ad}}}
    \newcommand{\Ad}{{\mathrm{Ad}}}\newcommand{\A}{{\mathrm{A}}}

    \renewcommand{\d}{{\mathrm{d}}}

    \newcommand{\End}{{\mathrm{End}}} 
    \newcommand{\el}{{\mathrm{ell}}}

    \newcommand{\f}{{\mathrm{f}}}
    
    \newcommand{\GSp}{{\mathrm{GSp}}}
    
    \newcommand{\Gal}{{\mathrm{Gal}}} 
    \newcommand{\GL}{{\mathrm{GL}}}
    \newcommand{\Hom}{{\mathrm{Hom}}}

    \newcommand{\Lie}{{\mathrm{Lie}}}
    \newcommand{\M}{{\mathrm{M}}}

    \renewcommand{\P}{{\mathrm{P}}}
     
    \newcommand{\pr}{{\mathrm{pr}}}

     \newcommand{\R}{{\mathrm{R}}}

    \newcommand{\reg}{{\mathrm{reg}}}
    \newcommand{\res}{{\operatorname{res}}}
    \newcommand{\rs}{{\mathrm{rs}}}

    \renewcommand{\ss}{{\mathrm{ss}}}
    \newcommand{\SL}{{\mathrm{SL}}}
    \newcommand{\Spec}{{\mathrm{Spec}}}

    \newcommand{\Supp}{{\mathrm{Supp}}}

    \newcommand{\tr}{{\mathrm{tr}}}
    \newcommand{\RTr}{{\mathrm{Tr}}}

    \newcommand{\vol}{{\mathrm{vol}}}

    \newcommand{\wt}{\widetilde}
    \newcommand{\wh}{\widehat}
    
    \newcommand{\pair}[1]{\langle {#1} \rangle}

    \newcommand{\incl}{\hookrightarrow}
    \newcommand{\sk}{\medskip}
    \newcommand{\lra}{\longrightarrow}
    \newcommand{\ra}{\rightarrow} 
    
    \newcommand{\bs}{\backslash}

    \newcommand{\s}{\sk\noindent}

    \newcommand{\abs}[1]{\lvert#1\rvert}

\begin{document}

\title{On linear periods}

\date{}
\maketitle

\begin{abstract}
Let $\pi'$ be a cuspidal automorphic representation of
$\GL_{2n}(\BA)$, which is assumed to be the Jacquet-Langlands
transfer from a cuspidal automorphic representation $\pi$ of
$\GL_{2m}(D)(\BA)$, where $D$ is a division algebra so that
$\GL_{2m}(D)$ is an inner form of $\GL_{2n}$. In this paper, we
consider the relation between linear periods on $\pi$ and $\pi'$. We
conjecture that the non-vanishing of the linear period on $\pi$
would imply the non-vanishing of that on $\pi'$. We illustrate an
approach using a relative trace formula towards this conjecture, and
prove the existence of smooth transfer over non-archimedean local
fields.

\end{abstract}

\section{Introduction}\label{section. intro}

\paragraph{Goal of this article}
Let $k$ be a number field, $\BA$ its ring of adeles, and $D$ a
central division algebra over $k$ of index $d$, that is,
$\dim_kD=d^2$. Let $\FG=\GL_{2m}(D)$, viewed as an algebraic group
over $k$, which is an inner form of $\FG'=\GL_{2n}$ with $n=md$. Let
$\pi$ be an irreducible cuspidal automorphic representation of
$\FG(\BA)$, and $\pi'$ the irreducible automorphic representation of
$\FG'(\BA)$ associated to $\pi$ by the Jacquet-Langlands
correspondence, which is assumed to be cuspidal. For the
Jacquet-Langlands correspondence involving general linear group and
its inner forms, we refer to \cite{dkv}, \cite{ba} and \cite{br} for
more details. The main purpose of this paper is to investigate a
relation between certain automorphic periods under the
Jacquet-Langlands correspondence.

To be more precise, let $\FZ$ be the center of $\FG$, which is
identified with the center $\FZ'$ of $\FG'$ via the obvious
identifications of $\FZ$ and $\FZ'$ with $\BG_m$ over $k$. Let
$\FH=\GL_m(D)\times\GL_m(D)$ (resp. $\FH'=\GL_n\times\GL_n$) be
embedded into $\FG$ (resp. $\FG'$) diagonally. The periods
considered in this paper are given by
$$\ell(\phi):=\int_{\FH(k)\FZ(\BA)\bs\FH(\BA)}\phi(h)\ \d h,
\quad \phi\in\pi,$$ and
$$\ell'(\varphi):=\int_{\FH'(k)\FZ'(\BA)\bs\FH'(\BA)}\varphi(h)\ \d h,
\quad \varphi\in\pi'.$$ We call them {\em linear periods}. In the
context of general linear groups (hence applying to $(\FG',\FH')$
above), this notion was introduced by \cite{fj}. We say that $\pi$
is $\FH$-distinguished or has a linear period if $\ell|_\pi\neq0$.
Of course, as a special case we get an analogous definition in the
context of $(\FG',\FH')$. Conjecturally, such a period has a close
relation with an L-value. For instance, it was shown in \cite{fj}
that $\pi'$ is $\FH'$-distinguished if and only if the L-value
$L^S(\frac{1}{2},\pi')\res_{s=1}L^S(s,\pi',\wedge^2)$ is nonzero,
using an integral representation of the L-function
$L^S(s_1,\pi')\cdot L^S(s_2,\pi',\wedge^2)$. What happens if $\pi$
is $\FH$-distinguished? The partial L-functions attached to $\pi$
and $\pi'$ should be the same, while there is no integral
representation for the ones associated to $\pi$. However, since
$\FH$ is an inner form of $\FH'$, it is natural to make the
following conjecture.

\begin{conj}\label{conj. main}
If $\pi$ is $\FH$-distinguished, then $\pi'$ is
$\FH'$-distinguished.
\end{conj}

\begin{remark}
As pointed out by D. Prasad, the converse of the above conjecture
should also hold. In other words, if $\pi'$ is $\FH'$-distinguished,
then $\pi$ should be $\FH$-distinguished too. Moreover, Conjecture 2
of \cite{pt} may be viewed as the local analog of Conjecture
\ref{conj. main} together with its converse.
\end{remark}

In this paper, we illustrate an approach towards this conjecture
using a relative trace formula. One of the key steps in this
approach is establishing the existence of smooth transfer over the
non-archimedean places. This is accomplished by Theorem \ref{thm.
main1}. Note that there is no need to prove the fundamental lemma,
since $(\FG(k_v),\FH(k_v))\simeq(\FG'(k_v),\FH'(k_v))$ for almost
all places $v$. Roughly speaking, let $v$ be a finite place of $k$.
Then the smooth transfer at $v$ is a ``transfer" $\lambda_v$ from
$\CC_c^\infty(\FG(k_v))$ to $\CC_c^\infty(\FG'(k_v))$ (in fact a map
from $\CC_c^\infty(\FG(k_v))$ to a suitable quotient of
$\CC_c^\infty(\FG'(k_v))$) such that for any
$f\in\CC_c^\infty(\FG(k_v))$, the orbital integrals of $f$ and
$\lambda_v(f)$ ``match''. Here the orbits are those of
$\FH\times\FH$ on $\FG$ (resp., $\FH'\times\FH'$ on $\FG'$) by left
and right translation. See Section \ref{section. smooth transfer}
for a precise definition.

Our proof of the existence of smooth transfer is mainly inspired by
Wei Zhang's work \cite{zh1} on the smooth transfer conjecture for
the Jacquet-Rallis relative trace formula towards the global
Gan-Gross-Prasad conjecture for unitary groups, and by Waldspurger's
work \cite{wa95} \cite{wa97} on endoscopic transfer which inspired
\cite{zh1}. The first step is to reduce the question of the
existence of smooth transfer to Lie algebras, that is, to linearize
the question. The second step is to show that, roughly speaking, the
Fourier transform commutes with smooth transfer. We will use a
global method to show such a property.

\paragraph{Some related work}
Conjecture \ref{conj. main} is motivated by the conjecture of H.
Jacquet and K. Martin \cite{jm} on Shalika periods. We briefly
recall it. Now let $d=n$ and $\FG=\GL_2(D)$. Denote by $\FS$ the
Shalika subgroup of $\FG$. To review its definition, consider the
parabolic subgroup $\FP=\FM\FN$ of $\FG$, where $\FM\simeq
D^\times\times D^\times$ is the obvious Levi subgroup and $\FN\simeq
D$ is the unipotent radical. Let $\psi$ be a nontrivial character of
$\BA/k$, which defines a nondegenerate character (still denoted by
$\psi$) of $\FN(k)\bs\FN(\BA)$ given by $\psi(x):=\psi(\tr_D(x))$
for $x\in\FN(\BA)\simeq D(\BA)$, where $\tr_D$ is the reduced trace
map on $D$. Then its stabilizer in $\FP$ is the Shalika subgroup
$\FS=\FL\FN$, where $\FL$ is $\Delta D^\times$ (i.e., $D^\times$
embedded diagonally in $\FM\simeq D^\times\times D^\times$). We can
extend $\psi$ to a character of $\FS(k)\bs\FS(\BA)$ by $\psi(l\cdot
n)=\psi(n)$ for $l\in\FL(\BA)$ and $n\in\FN(\BA)$. One can define
the Shalika subgroup $\FS'$ of $\FG'$ similarly, where the
corresponding parabolic subgroup is $\FP'=\FM'\FN'$ with Levi factor
$\FM'\simeq\GL_n\times\GL_n$. Then the Shalika period $\CS$ is a
linear form on $\pi$ given by
$$\CS(\phi)=\int_{\FS(k)\bs\FS(\BA)}\phi(u)\psi^{-1}(u)\ \d u,$$ and
the Shalika period $\CS'$ on $\pi'$ is defined similarly. In
\cite{jm}, Jacquet and Martin conjectured that if $\pi$ is
distinguished with respect to $\CS$ then $\pi'$ is also
distinguished with respect to $\CS'$. Under some hypotheses, using
relative trace formulae, Jacquet and Martin showed that this is true
if $n=2$. However, they did not prove the smooth transfer for the
full space $\CC_c^\infty(\FG(k_v))$ of Bruhat-Schwartz functions. Of
course, if one aims to completely prove this conjecture using the
method of the relative trace formula, one has to show the existence
of smooth transfer for the full space $\CC_c^\infty(\FG(k_v))$. In
the case $n=2$, this conjecture (together with its converse) was
completely proved by W. T. Gan and S. Takeda \cite{gt} using theta
correspondence. However, this method cannot be generalized to the
higher rank cases. Separately, D. Jiang, C. Nien and Y. Qin
\cite{jnq} proved this conjecture, under some conditions, for
general $n$ using the method of automorphic descent.

There is a relation between the linear period and the Shalika period
on $\pi'$. In fact, by the criterion for $\FH'$-distinction from
\cite{fj} recalled earlier, $\FH'$-distinction implies
$\FS'$-distinction, since $\pi'$ is $\FS'$-distinguished if and only
if the exterior L-function $L(s,\wedge^2,\pi')$ has a simple pole at
$s=1$. Locally, it was shown in \cite{jr} that if $\pi'_v$ is
$\FS'(k_v)$-distinguished then it was $\FH'(k_v)$-distinguished, and
it is conjectured that if $\pi'_v$ is generic then
$\FS'(k_v)$-distinction is equivalent to $\FH'(k_v)$-distinction.
Recently, Gan \cite{ga} proved this local conjecture using local
theta correspondence for dual pairs of type II. Therefore one can
ask whether there are such relations between linear and Shalika
periods on $\pi$, both globally and locally. Such a conjectural
relation together with the conjecture of Jacquet and Martin
motivates Conjecture \ref{conj. main}.

As we have said before, our proof of the existence of smooth
transfer is inspired by \cite{zh1} and \cite{wa97}. However, there
are still some significant differences between our method and that
of either of \cite{zh1} or \cite{wa97}. It is fair to say that ours
is a combination of theirs. We follow \cite{zh1} in reducing the
question of smooth transfer at the level of groups to showing
Theorem \ref{thm. fourier}, namely, the assertion that the Fourier
transform commutes with smooth transfer (up to an explicit
constant). However, we could not follow \cite{zh1} for the rest of
the proof, since the absence of a suitable partial Fourier transform
in our situation meant that the inductive arguments in
\cite[\S4]{zh1} could not be applied. We follow \cite{wa97} in using
a global method to prove Theorem \ref{thm. fourier}. This requires
us to study harmonic analysis on the corresponding $p$-adic
symmetric spaces, and prove several results analogous to ones
appearing in \cite{wa95} and \cite{wa97}, and others that are
analogues of more classical results in \cite{hc1} and \cite{hc}. We
just state these results and explain them briefly, since they are
direct generalizations of those that have been proved in the case of
$(\FG',\FH')$ in \cite{zh}.

In \cite{zh}, we studied the relation between similar periods
(involving an additional twist with a character) for the symmetric
pairs $(\FG',\FH')$ and $(\FG,\FH)$, where $(\FG',\FH')$ is as
before and $(\FG,\FH)=(\GL_n(D),\GL_n(k')$ with $D$ being a
quaternion algebra over $k$ and $k'$ being a quadratic field
extension of $k$ included in $D$. However, in \cite{zh}, we could
prove only ``half'' of the property that the Fourier transform
commutes with smooth transfer, due to the fact that there are
``fewer'' regular semisimple orbits associated to the pair
$(\FG,\FH)$ than to $(\FG',\FH')$. We encounter a similar problem in
this paper, though, fortunately, it turns out that this hurdle can
be circumvented. The point is that we have a nice description for
the orbits of $(\FG',\FH')$ that can be matched with ones of
$(\FG,\FH)$.

Sometimes the existence of smooth transfer for functions belonging
to a proper subspace of $\CC_c^\infty(\FG(k_v))$ suffices to prove
partial results towards Conjecture \ref{conj. main}. This is the
case, for instance, in the work of Jacquet-Martin \cite{jm}.

\paragraph{Structure of this article}
In \S\ref{section. trace formula}, we introduce the relative trace
formulae considered in this paper, which are natural for the
conjecture concerned. The contents of this section are more or less
routine and informal. The main purpose of this section is to show
the motivation for the study of smooth transfer.

To factor the global linear periods into local ones, we need to
study the property of multiplicity one for the symmetric pair
$(\FG(k_v),\FH(k_v))$ at each place $v$ of $k$, or, in other words,
to study the space $\Hom_{\FH(k_v)}(\pi_v,\BC)$ for any irreducible
admissible representation $\pi_v$ of $\FG(k_v)$. If
$\dim\Hom_{\FH(k_v)}(\pi_v,\BC)\leq 1$ for each irreducible
admissible representation $\pi_v$ of $\FG(k_v)$ we call
$(\FG(k_v),\FH(k_v))$ a Gelfand pair. We have not been able to show
$(\FG(k_v),\FH(k_v))$ to be a Gelfand pair, but we can show that it
satisfies a weaker variant of that property, which is enough for our
purpose of factoring the global period. In \S\ref{section.
multiplicity one}, we systematically follow the approach developed
by \cite{ag1} to study questions of this kind, i.e. using
generalized Harish-Chandra descent to study
$\FH(k_v)\times\FH(k_v)$-invariant distributions on $\FG(k_v)$. This
will also be important further into the paper (\S\ref{section.
smooth transfer} and \S\ref{section. local orbital integrals}),
while studying smooth transfer.

In \S\ref{section. smooth transfer}, we introduce the notion of
smooth transfer explicitly, both for groups and Lie algebras. After
several reduction steps, we show the existence of smooth transfer
(Theorem \ref{thm. main1}) assuming Theorem \ref{thm. fourier} on
the commutativity of Fourier transform with transfer. Theorem
\ref{thm. fourier} is proved in \S\ref{section. local orbital
integrals}.

The main aim of \S\ref{section. local orbital integrals} is to prove
Theorem \ref{thm. fourier}. We first recall some results on $p$-adic
harmonic analysis developed in \cite{zh} and give their
generalizations to our situation. With the aid of these results, we
prove Theorem \ref{thm. fourier} at the end of this section.

\section{Notations and conventions}\label{section. notation}

\paragraph{Actions of algebraic groups}
Let $k$ be a number field or a $p$-adic field. Let $\pi$ be an
action of a reductive group $\FM$ on a smooth affine variety $\FX$,
all defined over $k$. Write $M=\FM(k)$ and $X=\FX(k)$. Recall that
for $x\in X$, we say that $x$ is $\FM$-semisimple or $M$-semisimple
if $\FM x$ is Zariski closed in $\FX$ (or, equivalently, if $k$ is
$p$-adic, $M x$ is closed in $X$ for the analytic topology). We say
$x$ is $\FM$-regular or $M$-regular if the stabilizer $\FM_x$ of $x$
has minimal dimension. We denote by $\FX_\rs(k)$ or $X_\rs$ (resp.
$\FX_\ss(k)$ or $X_\ss$) the set of $M$-regular semisimple (resp.
$M$-semisimple) elements in $X$. If $k$ is $p$-adic, we call an
algebraic automorphism $\tau$ of $\FX$ $\FM$-admissible if (i)
$\tau$ normalizes $\pi(M)$ and $\tau^2\in\pi(M)$, (ii) for any
closed $M$-orbit $O\subset X$, $\tau(O)=O$.

\paragraph{Analysis on $\ell$-spaces} Now let $k$ be a $p$-adic field.
For a locally compact totally disconnected topological space $X$, we
denote by $\CC_c^\infty(X)$ the space of locally constant and
compactly supported $\BC$-valued functions on $X$, and by $\CD(X)$
the space of distributions on $X$, namely, the dual of
$\CC_c^\infty(X)$. If there is an action of an $\ell$-group $M$ on
$X$, we denote by $\CD(X)^M$ the subspace of $\CD(X)$ consisting of
$M$-invariant distributions.

\paragraph{Fourier transform and Weil index}
Now suppose that $k$ is a local field of characteristic 0. Let $X$
be a finite dimensional vector space over $k$ with the natural
topology induced from that of $k$, $\psi$ a nontrivial continuous
additive character of $k$, and $q$ a nondegenerate quadratic form on
$X$. We always equip $X$ with the self-dual Haar measure with
respect to the bi-character $\psi(q(\cdot,\cdot))$. Define the
Fourier transform $f\mapsto\wh{f}$ of the space $\CS(X)$ of
Bruhat-Schwartz functions on $X$ by
$$\wh{f}(x)=\int_X f(y)\psi(q(x,y))\ \d y.$$ Then
$\hat{\hat{f}}(x)=f(-x)$. We write $\gamma_\psi(q)$ for the Weil
index associated to $q$ and $\psi$, which is an 8th root of unity.
For the definition and some properties of the Weil index, see
\cite{we}.

\section{Relative trace formulae}\label{section. trace formula}
Let $(\FG,\FH)$ and $(\FG',\FH')$ be as defined in \S\ref{section.
intro}. Fix a Haar measure on $\FZ(\BA)$. For
$f\in\CC_c^\infty(\FG(\BA))$, define the kernel function
$$K_f(x,y)=\int_{\FZ(k)\bs\FZ(\BA)}\sum_{\gamma\in\FG(k)}
f(zx^{-1}\gamma y)\ \d z.$$ We consider the partially defined
distribution on $\FG(\BA)$
$$I(f)=\int_{\FH(k)\FZ(\BA)\bs\FH(\BA)}\int_{\FH(k)\FZ(\BA)\bs\FH(\BA)}
K_f(h_1,h_2)\ \d h_1\ \d h_2,$$ defined on the subspace of all
$f\in\CC_c^\infty(\FG(\BA))$ such that the above expression is
absolutely convergent. Choose the Haar measure on $\FZ'(\BA)$ to be
compatible with that on $\FZ(\BA)$. For
$f'\in\CC_c^\infty(\FG'(\BA))$ we define the kernel function
$K_{f'}(x,y)$ similarly. In the same way, we obtain a partially
defined distribution $J(\cdot)$ on $\FG'(\BA)$. The art of relative
trace formula is to compare $I(\cdot)$ with $J(\cdot)$.

Very informally, we have two ways to decompose $I(f)$ - the so
called spectral expansion and the so called geometric expansion. On
the spectral side, spherical characters $I_\pi(f)$ associated to
irreducible cuspidal representations $\pi$ of $\FG(\BA)$ are
involved. On the geometric side, orbital integrals $I_\gamma(f)$
associated to $\FH(k)\times\FH(k)$-regular semisimple orbits
$\gamma$ in $\FG(k)$ are involved. We give precise definitions of
$I_\pi(f)$ and $I_\gamma(f)$ below. Similarly, $J(f')$ can be
decomposed in these two ways.

Fix a Haar measure on $\FH(\BA)$. If $\pi$ is an irreducible
cuspidal representation of $\FG(\BA)$, let
$$K_{\pi,f}(x,y)=\sum_\varphi\left(\pi(f)\varphi\right)(x)\overline{\varphi(y)},$$
where $\varphi$ runs over an orthonormal basis for the space of
$\pi$. Define the spherical character $I_\pi$ to be
$$I_\pi(f)=\int_{\FH(k)\FZ(\BA)\bs\FH(\BA)}\int_{\FH(k)\FZ(\BA)\bs\FH(\BA)}
K_{\pi,f}(h_1,h_2)\ \d h_1\ \d h_2,$$ where
$f\in\CC_c^\infty(\FG(\BA))$. Both $K_{\pi,f}(x,y)$ and $I_\pi(f)$
are well defined, and we refer the reader to \cite[\S5]{hh} for a
detailed explanation. Thus, by definition, we have
$$I_\pi(f)=\sum_\varphi\ell\left(\pi(f)\varphi\right)\overline{\ell(\varphi)}.$$
Notice that $I_\pi$ is a distribution of positive type. In other
words, if $f=f_1*f_1^*$ where $f_1^*(g):=\bar{f}_1(g^{-1})$, then
$$I_\pi(f)=\sum_\varphi\ell\left(\pi(f_1)\varphi\right)
\overline{\ell(\pi(f_1)\varphi)}.$$ Hence $\pi$ is
$\FH$-distinguished if and only if $I_\pi$ is nonzero as a
distribution on $\FG(\BA)$. Therefore, the spectral expansion of
$I(f)$ in terms of $I_\pi(f)$ can reflect the property of
$\FH$-distinction. Similarly, we define the spherical character
$J_{\pi'}(f')$ associated to an irreducible cuspidal representation
$\pi'$ of $\FG'(\BA)$.

It is believed that, if we can compare the distributions $I$ and $J$
in some ways, $I_\pi$ and $J_{\pi'}$ are closely related, where
$\pi'$ is the Jacquet-Langlands correspondence of $\pi$. To compare
$I$ with $J$, we consider their geometric expansions.

If $\gamma\in\FG(k)$ is $\FH\times\FH$-regular semisimple, we fix a
Haar measure $\d_\gamma h$ on $\FH_\gamma(\BA)$, where
$$\FH_\gamma=\{ (h_1,h_2)\in\FH\times\FH;\ h_1\gamma h_2^{-1}=\gamma\}$$
is the stabilizer of $\gamma$ under the action of $\FH\times\FH$.
For $f\in\CC_c^\infty(\FG(\BA))$, define the orbital integral of $f$
at $\gamma$ to be
$$I_\gamma(f)=\int_{\FH_\gamma(\BA)\bs(\FH(\BA)\times\FH(\BA))}
f(h_1^{-1}\gamma h_2)\ \d h_1\ \d h_2.$$ This is well defined, since
the semisimple orbit is closed in $\FG(\BA)$ and therefore the above
integral is absolutely convergent. We fix a Haar measure $\d h_v$ on
$\FH(k_v)$ at each place $v$ of $k$ so that $\vol(\FH(\CO_{k_v}))=1$
for each unramified place $v$ and set $\d h=\prod_{v}\d h_v$. We
also fix a Haar measure $\d_\gamma h_v$ on $\FH_\gamma(k_v)$ at each
place $v$ of $k$ so that $\vol(\FH_\gamma(\CO_{k_v}))=1$ for each
unramified place $v$ and set $\d_\gamma h=\prod_v\d_\gamma h_v$. If
$f=\otimes'f_v$ is a pure tensor, we have
\begin{equation}\label{equation. global and local integrals 1}
I_\gamma(f)=\prod_v\int_{\FH_\gamma(k_v)\bs(\FH(k_v)\times\FH(k_v))}
f_v(h_1^{-1}\gamma h_2)\d h_1\ \d h_2,\end{equation} since the
integrals in the product are absolutely convergent and equal to 1 at
almost all places (cf. \cite[Proposition 12.21]{ge}). For
$f_v\in\CC_c^\infty(\FG(k_v))$, set
$$O_\gamma(f_v)=\int_{\FH_\gamma(k_v)\bs(\FH(k_v)\times\FH(k_v))}
f_v(h_1^{-1}\gamma h_2)\d h_1\ \d h_2.$$ Of course, the discussions
above contain the case $(\FG',\FH')$. We define $J_\delta(f')$ for
$\FH'\times\FH'$-regular semisimple elements $\delta\in\FG'(k)$ and
$f'\in\CC_c^\infty(\FG'(\BA))$, and define $O_\delta(f'_v)$ for
$f'_v\in\CC_c^\infty(\FG'(k_v))$ in the same way. Thus we have the
relation
\begin{equation}\label{equation. global and local integrals 2}
J_\delta(f')=\prod_v O_\delta(f'_v)\end{equation} if
$f'=\otimes'f'_v$.

From now on, when we say ``regular semisimple'', we mean
``$\FH\times\FH$-regular semisimple'' or ``$\FH'\times\FH'$-regular
semisimple'' if there is no confusion. If $\gamma\in\FG(k)$ is
regular semisimple, there exists a regular semisimple
$\delta\in\FG'(k)$ matching $\gamma$ (see Proposition \ref{prop.
match orbits} for more details). It turns out that $\FH_\gamma$ is
isomorphic to $\FH'_\delta$ (see Remark \ref{matching stabilizer}).
We equip $\FH'_\delta(\BA)$ with the Haar measure compatible with
that on $\FH_\gamma(\BA)$. To compare the regular parts of the
distributions $I$ with $J$ on the geometric side, we need to show
the following conjecture on smooth transfer.

\begin{conj}\label{conj. global smooth transfer}
For each $f$ in $\CC_c^\infty(\FG(\BA))$ there exists $f'$ in
$\CC_c^\infty(\FG'(\BA))$ such that for each $\delta\in\FG'(k)_\rs$
$$\begin{array}{lll}J_{\delta}(f')=\left\{\begin{array}{ll}
\begin{aligned}I_\gamma(f),&\quad \mathrm{if\ there\ exists\ }
\gamma\in \FG(k)_\rs\ \mathrm{ matching\ }\delta,\\
0,&\quad \mathrm{otherwise}.
\end{aligned}
\end{array}\right.
\end{array}$$
\end{conj}

\begin{thm}\label{thm. global smooth transfer}
Suppose that $D$ is split at all archimedean places. Then Conjecture
\ref{conj. global smooth transfer} holds.
\end{thm}

\begin{proof}
This is a direct consequence of the relations (\ref{equation. global
and local integrals 1}) and (\ref{equation. global and local
integrals 2}) between global and local orbital integrals, and
Theorem \ref{thm. main1} on the existence of smooth transfer at the
non-archimedean places.
\end{proof}

\section{Multiplicity one}\label{section. multiplicity one}
The global linear period $\ell$ belongs to the space
$\Hom_{\FH(\BA)}(\pi,\BC)$. To factor it into local ones, we need to
study the space $\Hom_{\FH(k_v)}(\pi_v,\BC)$ for each place $v$ of
$k$. We expect the so-called multiplicity one property to hold at
each place $v$, that is, if $\pi_v$ is an irreducible admissible
representation of $\FG(k_v)$, then
$\dim\Hom_{\FH(k_v)}(\pi_v,\BC)\leq1$. If $(\FG(k_v),\FH(k_v))$
satisfies this multiplicity one property, we call it a Gelfand pair.
It was proved in \cite{jr} in the non-archimedean case and in
\cite{ag1} in the archimedean case that $(\FG'(k_v),\FH'(k_v))$ is a
Gelfand pair. When $m=1$, $v$ is non-archimedean and $D$ is a
general division algebra, $(\FG(k_v),\FH(k_v))$ is also a Gelfand
pair, as was proved by Prasad \cite{pr}. We have not been able to
show $(\FG(k_v),\FH(k_v))$ to be a Gelfand pair for general $m$ and
$D$. However, we can prove a weaker result which is enough for us to
factor the global period $\ell$, namely Proposition \ref{prop.
multiplicity} and Corollary \ref{cor. multiplicity one for unitary
repn}.

From now on, and until the end of the paper, we fix a $p$-adic field
$F$. We follow the same line as that of \cite{ag1} where an
effective way to prove results like multiplicity one for symmetric
pairs is systematically studied, and we refer the reader to
\cite{ag1} for more details.

\paragraph{Symmetric pairs}
Now let $D$ be a division algebra over $F$ of index $d$. Let $\FG$
and $\FH$ be as defined in \S\ref{section. intro} associated to $D$,
both viewed as algebraic groups over $F$ now. Write $G=\FG(F)$ and
$H=\FH(F)$. Let
$\epsilon=\begin{pmatrix}{\bf1}_m&0\\0&-{\bf1}_m\end{pmatrix}$ and
define an involution $\theta$ on $\FG$ by $\theta(g)=\epsilon
g\epsilon$. Then $\FH=\FG^\theta$, that is, $(\FG,\FH,\theta)$ is a
symmetric pair. When there is no confusion, we write $(\FG,\FH)$
instead of $(\FG,\FH,\theta)$ for simplicity. Let $\iota$ be the
anti-involution on $\FG$ defined by $\iota(g)=\theta(g^{-1})$. Write
$\FG^\iota=\{g\in\FG;\ \iota(g)=g\}$ and define a symmetrization map
$$s:\FG\lra\FG^\iota,\ s(g)=g\iota(g).$$
Via this map we can identify the $p$-adic symmetric space $S=G/H$
with its image in $\FG^\iota(F)$. Since $H^1(F,\FH)$ is trivial, we
also have $S=(\FG/\FH)(F)$.

Let $\theta$ act by its differential on $\fg=\Lie(\FG)$. Write
$\fh=\Lie(\FH)$. Thus, $\fh=\{X\in\fg;\ \theta(X)=X\}$. Write
$\fs=\{X\in\fg;\ \theta(X)=-X\}$. $\fs$ can be viewed as a sort of
``Lie algebra" for $\FG/\FH$, on which $\FH$ acts by adjoint action.
This action can be described more concretely as follows. It is easy
to see that $\fs\simeq\fg\fl_m(D)\oplus\fg\fl_m(D)$, and modulo this
isomorphism, the action of $\FH$ on $\fs$ is given by
$$(h_1,h_2)\cdot(X_1,X_2)=(h_1X_1h_2^{-1},h_2X_2h_1^{-1})$$ for
$(h_1,h_2)\in\FH$ and $(X_1,X_2)\in\fs$.

We fix the nondegenerate symmetric bilinear form $\pair{\ ,\ }$ on
$\fg(F)$ defined by
$$\pair{X,Y}=\tr(XY),\quad\mathrm{for}\ X,Y\in\fg(F),$$ where $\tr$
is the reduced trace map on $\fg\fl_{2m}(D)$, identified with the
space $\End_D(D^{2m})$. Notice that the form $\pair{\ ,\ }$ is both
$G$-invariant and $\theta$-invariant. When we want to emphasize the
index $m$, we write $\FG_m,\FH_m,\theta_m,\fs_m$. Notice that the
case $m=n$ and $D=F$ is just the case denoted by $(\FG',\FH')$ in
\S\ref{section. intro}.

We will consider the action of $H\times H$ on $G$ by left and right
translation, and the adjoint action of $H$ on $S$ or $\fs(F)$. These
actions are related by $$s\left((h_1,h_2)\cdot g\right)=h_1\cdot
s(g),$$ for $(h_1,h_2)\in H\times H$ and $g\in G$.

Now we recall some notions attached to a general symmetric pair
$(\FG,\FH,\theta)$. We refer the reader to \cite{ag1} and \cite{ag2}
for more details. Define $Q(\fs)=\fs/\fs^\FH$. Since $\FH$ is
reductive, there exists a unique $\FH$-equivariant splitting
$Q(\fs)\incl\fs$. Denote by $\phi:\fs\ra\fs/\FH$ the standard
projection. Let $\CN$ be the set of elements of $\fs$ that belong to
the null-cone of $\fg$. We call $\CN$ the null-cone of $\fs$, since
$\CN=\phi^{-1}(\phi(0))$ by \cite[Lemma 7.3.8]{ag1}. Note that
$\CN\subset Q(\fs)$. Let $R(\fs)=Q(\fs)-\CN$. We call an element
$g\in G$ \emph{admissible} if (i) $\Ad(g)$ commutes with $\theta$
and (ii) $\Ad(g)|_\fs$ is $\FH$-admissible. Notice that, in our
case, $Q(\fs)=\fs$.

A symmetric pair $(\FG,\FH,\theta)$ is called \emph{good}, if for
any closed $H\times H$ orbit $O$ in $G$, $\iota(O)=O$.

A symmetric pair $(\FG,\FH,\theta)$ is called \emph{regular} if for
any admissible $g\in G$ such that
$\CD(R(\fs)(F))^H\subset\CD(R(\fs)(F))^{\Ad(g)}$ we have
$$\CD(Q(\fs)(F))^H\subset\CD(Q(\fs)(F))^{\Ad(g)}.$$

For each nilpotent element $X\in\fs(F)$, there exists a group
homomorphism $\phi:\SL_2(F)\ra G$ such that $X=\d
\phi\left(\begin{pmatrix}0&1\\0&0\end{pmatrix}\right)$, $Y:=\d
\phi\left(\begin{pmatrix}0&0\\1&0\end{pmatrix}\right)$ belongs to
$\CN$, and ${\bf d}(X):=\d
\phi\left(\begin{pmatrix}1&0\\0&-1\end{pmatrix}\right)$ belongs to
$\fh(F)$ and is semisimple (cf. \cite[Lemma 7.1.11]{ag1}). We call
$(X,{\bf d}(X),Y)$ an $\fs\fl_2$-triple. Such a triple depends on
the choice of $\phi$. However, the choice does not really matter
(cf. \cite[Notation 7.1.12]{ag1}).

We say that a symmetric pair $(\FG,\FH,\theta)$ is of \emph{negative
defect} if for any nilpotent $X\in\fs(F)$ we have
$$\RTr(\ad({\bf d}(X))|_{\fh_X})<\dim Q(\fs),$$ where $\fh_X$ is the
centralizer of $X$ in $\fh$. By \cite[Proposition 7.3.5, Proposition
7.3.7, Remark 7.4.3]{ag1}, we see that if a symmetric pair is of
negative defect it is regular.

Let $g\in G$ be $H\times H$-semisimple and $x=s(g)$. Then $x$ is
semisimple both as an element of $G$ and with respect to the
$H$-action (cf. \cite[Proposition 7.2.1]{ag1}). The triple
$(\FG_x,\FH_x,\theta|_{\FG_x})$ is still a symmetric pair (clearly
$\theta$ preserves $\FG_x$ for $x\in\FG^\iota(F)$). A symmetric pair
obtained this way is called a descendant of $(\FG,\FH,\theta)$.
Notice that the ``Lie algebra'' of $\FG_x/\FH_x$ can be identified
with $\fs_x$, where $\fs_x$ is the set of elements in $\fs$
commuting with $x$. It was shown in \cite[Theorem 7.4.5]{ag1} that
if $(\FG,\FH,\theta)$ is a good symmetric pair such that all its
descendants are regular then it is a GK-pair. Here, the statement
that $(\FG,\FH)$ is a GK-pair means:
$$\CD(G)^{H\times H}\subset\CD(G)^\iota.$$

\paragraph{Descendants}
Now we return to the specific symmetric pair that concerns us in the
paper. To study the property of multiplicity one, as we have
explained, it is important to know all descendants of
$(\FG,\FH,\theta)$. The following proposition gives a list of all
possible descendants.
\begin{prop}\label{prop. descendant 0}
All descendants of $(\FG,\FH,\theta)$ are products of symmetric
pairs of the following types
\begin{enumerate}
\item
$\left(\R_{L/F}\left(\GL_r(D')\times\GL_r(D')\right),
\Delta(\R_{L/F}\GL_r(D')),\delta\right)$ for some field extension
$L/F$ and some central division algebra $D'$ over $L$,
\item
$\left(\R_{L'/F}\GL_r(D'\otimes_LL'),\R_{L/F}\GL_r(D'),\gamma\right)$
for some field extension $L/F$, a quadratic extension $L'/L$ and
some central division algebra $D'$ over $L$,
\item $(\GL_{2r}(D),\GL_r(D)\times\GL_r(D),\theta)$.
\end{enumerate}
\end{prop}

\begin{remark}
Here we use $\Delta$ to denote the diagonal embedding and use
$\R_{L/F}$ to denote the Weil restriction with respect to the field
extension $L/F$. The involution $\delta$ in (1) of the above
proposition is $(x,y)\mapsto(y,x)$, $\gamma$ in (2) is induced by
the nontrivial element of $\Gal(L'/L)$, and $\theta$ in (3) is the
one introduced before.
\end{remark}

\begin{proof}
In the case $D=F$, this proposition was first proved in
\cite[Proposition 4.3]{jr} and reproved in \cite[Theorem
7.7.3]{ag1}. Our proof is similar to that of \cite[Theorem
7.7.3]{ag1}.

Let $x\in\FG^\iota(F)$ be a semisimple element. Put $V=D^{2m}$ and
view $x$ as an element of $\GL_D(V)$. Let $m(t)=\prod_{i=1}^sp_i(t)$
be the minimal polynomial of $x$, where $p_i(t)\in F[t]$ is a monic
irreducible polynomial and $p_i\neq p_j$ if $i\neq j$. Set
$L_i:=F[t]/(p_i(t))$, which can be viewed as a field extension of
$F$. Then $F[x]\simeq\prod_{i=1}^s L_i$. $V$ is an
$(F[x],D)$-bimodule, and has a decomposition $V=\bigoplus_{i=1}^s
V_i$ where $V_i$ is a $D$-submodule and $L_i$ acts faithfully on it.
Thus $V_i$ is a $D\otimes_FL_i$-module. Since $D\otimes_FL_i$ is a
central simple algebra over $L_i$, $D\otimes_F L_i=\M_{c_i\times
c_i}(D_i)$ for some central division algebra $D_i$ over $L_i$. Set
$V_i=W_i^{\oplus t_i}$ where $W_i\simeq D_i^{\oplus c_i}$. The above
discussion on linear algebra over $D$ can be found in
\cite[\S3]{yu}.

Therefore, $\FG_x\simeq\prod_{i=1}^s\R_{L_i/F}\GL_{t_i}(D_i)$. The
rest of the proof is the same as that of \cite[Theorem 7.7.3]{ag1}.
We remark that only the condition $x\in\FG^\iota(F)$ is used in the
proof of \cite[Theorem 7.7.3]{ag1}. By this condition, we can only
deduce a weaker result, namely that all descendants of
$(\FG,\FH,\theta)$ are products of symmetric pairs of types 1, 2 or
3, where a type 3 symmetric pair is of the form
$$(\GL_{q+r}(D),\GL_{q}(D)\times\GL_r(D),\theta_{q,r}),$$ where
$\theta_{q,r}(g)=\epsilon_{q,r}g\epsilon_{q,r}$ with
$\epsilon_{q,r}=\begin{pmatrix}{\bf 1}_q&0\\0&-{\bf
1}_r\end{pmatrix}$ and $q$ may not equal $r$. Since $x$ lies in
$S\subset\FG^\iota(F)$, the list of possibilities for $\fs_x$
computed in Proposition \ref{prop. descendant 1 group} below lets us
eliminate factors of the form
$(\GL_{q+r}(D),\GL_q(D)\times\GL_r(D),\theta_{q,r})$ with $q\neq r$.
\end{proof}

\paragraph{Multiplicity one}
From the above classification of the descendants, we can see that,
for any $H\times H$-semisimple $g\in G$, $H^1(F,\FH_{s(g)})$ is
trivial. By \cite[Corollary 7.1.5]{ag1}, this implies that the
symmetric pair $(\FG,\FH)$ is good.

This classification also implies that all the descendants of
$(\FG,\FH)$ are regular. The reason is that, after base change to
some extension field $F'$, they are of negative defect over $F'$
(proved in \cite[Theorem 7.6.5]{ag1} for symmetric pairs of types 1
and 2 and in \cite[Lemma 7.7.5]{ag1} for symmetric pairs of type 3),
and hence they are of negative defect over $F$ by \cite[Lemma
4.2.7]{ag2}.

Therefore $(\FG,\FH)$ is a GK-pair. In particular, by
\cite[Corollary 8.1.6]{ag1}, it satisfies the following property,
which is a weaker variant of the property defining Gelfand pairs.

\begin{prop}\label{prop. multiplicity}
For any irreducible admissible representation $\pi$ of $G$ we have
$$\dim\Hom_H(\pi,\BC)\cdot\dim\Hom_H(\wt{\pi},\BC)\leq 1,$$
where $\wt{\pi}$ is the contragredient of $\pi$.
\end{prop}

\begin{cor}\label{cor. multiplicity one for unitary repn}
If $\pi$ is an irreducible unitary admissible representation of $G$,
we have $$\dim\Hom_H(\pi,\BC)\leq1.$$
\end{cor}

\begin{proof}
The following ``well known'' argument was pointed out by Prasad to
the author. If $\pi$ is unitary, then $\bar{\pi}=\wt{\pi}$, where
$\bar{\pi}$ denotes the complex conjugate of $\pi$. Observe that if
$\pi$ has an $H$-invariant form $\ell$, taking the complex conjugate
of the form, one can obtain an $H$-invariant form $\bar{\ell}$ on
$\bar{\pi}\simeq\wt{\pi}$.

\end{proof}

\begin{remark} In general, we expect that $(\FG,\FH)$ is a Gelfand pair.
When $D=F$, this is the main theorem of \cite{jr}. For $m=1$ and
general $D$, it was proved by Prasad in \cite[\S7]{pr}. For general
$m$ and $D$, we do not know how to prove it, because we do not know
a Gelfand-Kazhdan type realization for the contragredient
representation of an irreducible admissible representation of $G$.
However, if $D$ is a quaternion algebra, there is an
anti-automorphism $\tau$ of $G$ such that $\tau^2\in\Ad(G(F))$,
$\tau$ preserves any closed conjugacy class in $G(F)$, and
$\tau(H)=H$ (cf. \cite[Theorem 3.1]{ra}). Thus, by \cite[Proposition
2.1.6]{ag2}, we have the following result.
\end{remark}

\begin{cor}
If $D$ is a quaternion algebra, for any irreducible admissible
representation $\pi$ of $G$ we have $$\dim\Hom_H(\pi,\BC)\leq1.$$
\end{cor}

\begin{remark}\label{rem. archimedean}
The results of this section also hold when $F=\BR$, and they can be
proved by the same arguments. All the notions we have introduced and
all the propositions and theorems we have quoted hold in the
archimedean case.
\end{remark}

\section{Smooth transfer}\label{section. smooth transfer}
We keep the notations as before, and continue to let $F$ be a
$p$-adic field. In this section, we will show the existence of the
smooth transfer with respect to the relative trace formulae
concerned in this paper. Our strategy here follows the somewhat
standard procedure, that was also used to study smooth transfer in
other cases (cf. \cite{zh1} or \cite{zh} for more details). This
strategy arises from the work of Waldspurger on endoscopic transfer
(cf. \cite{wa97}). After several reduction steps, we reduce the
question to proving the property that ``the Fourier transform
commutes with smooth transfer''. We refer the reader to Theorem
\ref{thm. fourier} for the exact statement. This property will be
proved in Section \ref{section. local orbital integrals}.

The local orbital integrals that we consider are
\begin{equation}\label{equation. local orbital
1}O(g,f)=\int_{H_g\bs(H\times H)}f(h_1^{-1}gh_2)\ \d h_1\ \d
h_2,\end{equation} where $g\in G$ is $H\times H$-regular semisimple,
and $f\in\CC_c^\infty(G)$. The quotient map $q:G\ra G/H=S$ gives
rise to a surjection $\wt{q}:\CC_c^\infty(G)\ra\CC_c^\infty(S)$
defined by $$(\wt{q}f)(\bar{y})=\int_Hf(yh)\ \d h$$(cf. \cite[Lemma
in Section 1.21]{bz}). Let $\wt{f}=\wt{q}(f)$. We identify $S$ with
its image in $\FG^\iota(F)$ under the symmetrization map $s$, and
view $\wt{f}$ as a $\CC_c^\infty$-function on the image of $S$.
Thus, by definition, $\wt{f}(x)=\wt{f}(\bar{g})$ if $x=s(g)$. Now
let $g\in G$ be $H\times H$-regular semisimple. Then $x=s(g)$ is
$H$-regular semisimple (cf. \cite[Proposition 7.2.17]{ag1}) and we
have the relations
$$H_g\bs(H\times H)\twoheadrightarrow H_g(\{1\}\times H)\bs(H\times
H)\stackrel{\pr_1}{\simeq}H_x\bs H,$$ and $H_g\cap(\{1\}\times
H)=\{1\}\times\{1\}$. Therefore we have
$$\begin{aligned}O(g,f)
&=\int_{H_g(\{1\}\times H)\bs(H\times H)}\int_{\{1\}\times
H}f(h_1^{-1}xh_2)\ \d h_2\ \d h_1\\
&=\int_{H_x\bs H}\wt{f}(h^{-1} x h)\ \d h.
\end{aligned}$$ Henceforth it suffices to consider the orbital integrals for
$\CC_c^\infty(S)$ with respect to the $H$-action. Eventually, we
also have to consider the orbital integrals for
$\CC_c^\infty(\fs(F))$ with respect to the $H$-action.

\paragraph{Orbits}
First, we classify all $H$-semisimple orbits of $S$ and $\fs(F)$. We
remark that the results here on the orbits also hold when we replace
$F$ by a number field $k$. For each semisimple element $x$ in $S$ or
$\fs(F)$, it is also important to determine the couple
$(\FH_x,\fs_x)$ which is also called the descendant of $(\FH,\fs)$
at $x$. The reason for considering a semisimple element $x$ of $S$
(resp. $\fs(F)$) and the descendant of $(\FH,\fs)$ at $x$ is the
following. We can reduce the study of the orbital integrals of an
element $f$ in $\CC_c^\infty(S)$ (resp. $\CC_c^\infty(\fs(F))$) at
regular semsimple elements near $x$ to the study of orbital
integrals for the $H_x$-action of an appropriate element $f^\#_x$ in
$\CC_c^\infty(\fs_x(F))$ at regular semisimple elements close to 0
in $\fs_x(F)$. Here $f^\#_x$ is obtained from $f$ by a process
called semisimple descent. We refer the reader to \cite[Proposition
3.11]{zh1} for more details.

\begin{prop}\label{prop. descendant 1 group}
\begin{enumerate}
\item Each semisimple element $x$ of $S$ is $H$-conjugate to an element of
the form
$$x(A,m_1,m_2):=\begin{pmatrix}A&0&0&{\bf1}_r&0&0\\
0&{\bf1}_{m_1}&0&0&0&0\\
0&0&-{\bf1}_{m_2}&0&0&0\\
A^2-{\bf1}_r&0&0&A&0&0\\
0&0&0&0&{\bf1}_{m_1}&0\\
0&0&0&0&0&-{\bf1}_{m_2}\end{pmatrix},$$ with $m=r+m_1+m_2$, $A\in
\fg\fl_r(D)$ being semisimple without eigenvalues $\pm1$ and unique
up to conjugation. Moreover, $x(A,m_1,m_2)$ is regular if and only
if $m_1=m_2=0$ and $A$ is regular in $\fg\fl_m(D)$.
\item Let $x=x(A,m_1,m_2)$ in $S$ be semisimple. Then the descendant
$(\FH_x,\fs_x)$ is isomorphic to the product
$$\left(\GL_r(D)_A,\fg\fl_r(D)_A\right)\times(\FH_{m_1},\fs_{m_1})
\times(\FH_{m_2},\fs_{m_2}).$$ Here $\GL_r(D)_A$ and $\fg\fl_r(D)_A$
are the centralizers of $A$ in $\GL_r(D)$ and $\fg\fl_r(D)$
respectively, and $\GL_r(D)_A$ acts on $\fg\fl_r(D)_A$ by the
adjoint action.
\end{enumerate}
\end{prop}

\begin{proof}

The first assertion was proved in \cite[Proposition 4.1]{jr} in the
case $D=F$. The reader can check that the same proof goes through
for general $D$ without difficulty. We only provide some steps in
the proof.

Let $x=\begin{pmatrix}A&B\\P&Q\end{pmatrix}$ be a semisimple element
of $S$ inside $\FG^\iota(F)$, with $A,B,P,Q$ in $\fg\fl_m(D)$. We
claim that $A,Q,BP,PB$ and $\begin{pmatrix}0&B\\P&0\end{pmatrix}$
are semisimple matrices. This is the case when $D=F$ by \cite[Lemma
4.2]{jr}.  Actually, in the proof of \cite[Lemma 4.2]{jr}, one can
assume that $F$ is algebraically closed, since the condition of
being semisimple does not depend on the ground field. Hence the
claim for any general $D$ follows. Since $x\in\FG^\iota(F)$, we have
the relations
\begin{equation}\label{equation. matrix condition}
A^2={\bf 1}_m+BP,\quad Q^2={\bf 1}_m+PB,\quad AB=BQ,\quad QP=PA.
\end{equation}
Replacing $x$ by a conjugate under $H$, we may assume
$B=\begin{pmatrix}{\bf 1}_r&0\\0&0\end{pmatrix}$. Since
$\begin{pmatrix}0&B\\P&0\end{pmatrix}$ is semisimple, by
\cite[Proposition 2.1]{jr} (which is valid here), $x$ is
$H$-conjugate to an element of the form
$\begin{pmatrix}A&\begin{pmatrix}{\bf 1}_r&0\\0&0\end{pmatrix}\\
\begin{pmatrix}C_r&0\\0&0\end{pmatrix}&D\end{pmatrix}$ where
$C_r\in\GL_r(D)$ is semisimple. The relations (\ref{equation. matrix
condition}) force such an $H$-conjugate of $x$ to be of the form
$\begin{pmatrix}A&0&{\bf1}_r&0\\
0&A'&0&0\\
A^2-{\bf1}_r&0&A&0\\
0&0&0&Q'\end{pmatrix}$, where $A\in\fg\fl_r(D)$ is semisimple
without eigenvalues $\pm1$, $A'$ and $Q'$ are elements of order 2.
By the same discussion as in the proof of \cite[Proposition
4.1]{jr}, we see that $x$ is $H$-conjugate to some $x(A,m_1,m_2)$.

Conversely, by similar discussions as in the proofs of \cite[Lemma
4.3, Proposition 4.1]{jr}, we see that each $x(A,m_1,m_2)$ is
semisimple and lies in $S$. The remaining assertions such as the
uniqueness of $A$ up to conjugation are straightforward.

For the second assertion, since descendants $(\FH_x,\fs_x)$ can be
obtained by computing the centralizers of $x$ in $\FH$ and $\fs$, we
leave the details to the reader.

\end{proof}

\begin{prop}\label{prop. descendant 1 lie}
\begin{enumerate}
\item Each semisimple element $X$ of $\fs(F)$ is $H$-conjugate to an
element of the form
$$X(A)=\begin{pmatrix}0&0&{\bf1}_r&0\\0&0&0&0\\A&0&0&0\\0&0&0&0\end{pmatrix}$$
with $A\in\GL_r(D)$ being semisimple and unique up to conjugation.
Moreover, $X(A)$ is regular if and only if $r=m$ and $A\in\GL_m(D)$
is regular.
\item Let $X=X(A)$ in $\fs(F)$ be semisimple. Then the descendant $(\FH_X,\fs_X)$
is isomorphic to the product
$$\left(\GL_r(D)_A,\fg\fl_r(D)_A\right)\times(\FH_{m-r},\fs_{m-r}).$$
\end{enumerate}
\end{prop}

\begin{proof}
This proposition was proved in \cite[Proposition 2.1, Proposition
2.2]{jr} in the case $D=F$. The proof is simpler than that of
Proposition \ref{prop. descendant 1 group} and can be applied to our
more general situation directly. We leave the details to the reader.
\end{proof}

\paragraph{Matching between the orbits}
Now we give a description for the matching between $H$-semisimple
orbits in $S$ or $\fs(F)$ and $H'$-semisimple orbits in $S'$ or
$\fs'(F)$. Let $L$ be a field extension of $F$ with degree $d$
contained in $D$. Since $D$ is a $d$-dimensional $L$-vector space,
we can obtain an embedding $D\incl\fg\fl_d(L)$. This induces
embeddings $(G,H)\incl(\FG'(L),\FH'(L))$, $S\incl \FG'(L)/\FH'(L)$
and $\fs(F)\incl \fs'(L)$. We identify $\FG'(L)/\FH'(L)$ with its
image in $\FG'^\iota(L)$.

\begin{prop}\label{prop. match orbits}
\begin{enumerate}
\item For each semisimple element $x$ of $S$, there
exists $h\in\FH'(L)$ such that $hxh^{-1}$ belongs to $S'$. This
establishes an injection from the set of $H$-semisimple orbits in
$S$ into the set of $H'$-semisimple orbits in $S'$. This injection
carries the orbit of $x(A,m_1,m_2)$ in $S$ to the orbit of
$x(B,m_1d,m_2d)$ in $S'$ such that $A\in\fg\fl_{m-m_1-m_2}(D)$ and
$B\in\fg\fl_{(m-m_1-m_2)d}(F)$ have the same characteristic
polynomial.
\item For each semisimple element $X$ of $\fs(F)$, there
exists $h\in\FH'(L)$ such that $hXh^{-1}$ belongs to $\fs'(F)$. This
establishes an injection from the set of $H$-semisimple orbits in
$\fs(F)$ into the set of $H'$-semisimple orbits in $\fs'(F)$. This
injection carries the orbit of $X(A)$ in $\fs(F)$ to the orbit of
$X(B)$ in $\fs'(F)$ such that $A\in\GL_{r}(D)$ and $B\in\GL_{rd}(F)$
have the same characteristic polynomial.
\end{enumerate}
\end{prop}

\begin{proof}
We only prove the matching between the orbits in $\fs(F)$ and
$\fs'(F)$. The proof of the matching between the orbits in $S$ and
$S'$ is similar. It is harmless to assume that $x$ is of the form
$X(A)$ with $A\in\GL_r(D)$. We view $A$ as an element in
$\GL_{rd}(L)$. Since the coefficients of the characteristic
polynomial of $A$ are in $F$, there exists $h_0\in\GL_{rd}(L)$ such
that $B=h_0^{-1}Ah_0$ is in $\GL_{rd}(F)$. Let
$h=\begin{pmatrix}h_0&&&\\&{\bf1}_{(m-r)d}&&\\&&h_0&\\&&&{\bf
1}_{(m-r)d}\end{pmatrix}$. Then $h^{-1}\cdot X(A)\cdot h=X(B)$.
\end{proof}

\begin{defn}
We say that $x\in S_\ss$ (resp. $X\in\fs_\ss(F)$) matches $y\in
S'_\ss$ (resp. $Y\in\fs'_\ss(F)$), and write $x\leftrightarrow y$
(resp. $X\leftrightarrow Y$), if the above map sends the orbit of
$x$ (resp. $X$) to the orbit of $y$ (resp. $Y$).
\end{defn}

Now we discuss some properties of the above matching. These
properties will be used in Section \ref{section. local orbital
integrals}.

\begin{remark}\label{rem. matching orbits}
For a regular semsimple element
$Y=\begin{pmatrix}0&A\\B&0\end{pmatrix}\in\fs'_\rs(F)$, suppose we
wish to know whether there exists $X\in\fs_\rs(F)$ such that
$X\leftrightarrow Y$. Note that $Y$ is $H'$-conjugate to $X(AB)$. It
is well known that there exists $C\in\GL_m(D)$ such that $C$ and
$AB$ have the same characteristic polynomial if and only if the
degree of every irreducible factor (over $F$) of the characteristic
polynomial of $AB$ can be divided by $d$. Hence there exists
$X\in\fs_\rs(F)$ such that $X\leftrightarrow Y$ if and only if the
degree of every irreducible factor (over $F$) of the characteristic
polynomial of $AB$ is divisible by $d$.
\end{remark}

\begin{remark}\label{rem. matching descendants}
Suppose that $x\in S_\ss$ and $y\in S'_\ss$ match. We want to
compare $(\FH_x,\fs_x)$ with $(\FH'_y,\fs'_y)$. It is harmless to
assume that $x=x(A,m_1,m_2)$ and $y=x(B,m_1d,m_2d)$. Then, according
to Proposition \ref{prop. descendant 1 group}
$$(\FH_x,\fs_x)\simeq\left(\GL_r(D)_A,\fg\fl_r(D)_A\right)\times(\FH_{m_1},\fs_{m_1})
\times(\FH_{m_2},\fs_{m_2}),$$ and
$$(\FH'_y,\fs'_y)\simeq\left(\GL_{rd,B},\fg\fl_{rd,B}\right)\times(\FH'_{m_1d},\fs'_{m_1d})
\times(\FH'_{m_2d},\fs'_{m_2d}).$$ Note that
$\left(\GL_r(D)_A,\fg\fl_r(D)_A\right)$ is an inner form of
$\left(\GL_{rd,B},\fg\fl_{rd,B}\right)$. Also note that the other
factors are related in a similar manner as $(\FH,\fs)$ and
$(\FH',\fs')$ are. For $X\in\fs_\ss(F)$ and $Y\in\fs'_\ss(F)$ such
that $X\leftrightarrow Y$, according to Proposition \ref{prop.
descendant 1 lie}, the descendants $(\FH_X,\fs_X)$ and
$(\FH'_Y,\fs'_Y)$ satisfy a similar relation as above.
\end{remark}

\begin{remark}\label{matching stabilizer}
Suppose that $x$ in $S_\rs$ (resp. $\fs_\rs(F)$) and $y$ in $S'_\rs$
(resp. $\fs'_\rs(F)$) match. By Propositions \ref{prop. descendant 1
group}, \ref{prop. descendant 1 lie} and \ref{prop. match orbits},
we have $$\FH_x\simeq\FH'_y.$$
\end{remark}

\begin{remark}\label{rem. matching cartan}
Recall that a Cartan subspace of $\fs$ is by definition a maximal
abelian (with respect to the Lie bracket on $\fg$) subspace
consisting of semisimple elements. An element $X\in\fs$ is regular
semisimple if the centralizer of $X$ in $\fs$ is a Cartan subspace
(cf. \cite[page 471]{Vi}). For a Cartan subspace $\fc$, we denote by
$\fc_\reg(F)$ the subset of regular elements in $\fc(F)$. For a
Cartan subspace $\fc$ of $\fs$ and a Cartan subspace $\fc'$ of
$\fs'$, we say that $\fc$ and $\fc'$ match and write
$\fc\leftrightarrow\fc'$ if there exist $X\in\fc_\reg(F)$ and
$Y\in\fc'_\reg(F)$ such that $X\leftrightarrow Y$. Note that if
$\fc\leftrightarrow\fc'$, there is an isomorphism
$\varphi_\fc:\fc\ra\fc'$ defined over $F$ such that
$X\leftrightarrow\varphi_\fc(X)$ for any $X\in\fc_\reg(F)$. To see
this, suppose that $X\in\fc_\reg(F)$ and $Y\in\fc'_\reg(F)$ match.
We may assume that $X=X(A)$ and $Y=Y(B)$. Then we have
$$\fc(F)=\fs_X(F)=\left\{\begin{pmatrix}0&C\\AC&0\end{pmatrix};C\in
\fg\fl_m(D),\ AC=CA\right\}\simeq\gl_{m}(D)_A.$$  We also have
$$\fc'(F)=\fs'_Y(F)=\left\{\begin{pmatrix}0&D\\BD&0\end{pmatrix};D\in
\fg\fl_n,\ BD=DB\right\}\simeq\fg\fl_{n,B}.$$ Since $A$ and $B$ are
regular semisimple and have the same characteristic polynomial,
there is an isomorphism $\varphi:\fg\fl_m(D)_A\ra\fg\fl_{n,B}$ over
$F$ such that $\varphi(A)=B$. The isomorphism $\varphi_c$ can be
obtained from $\varphi$.
\end{remark}

\paragraph{Smooth transfer}
Now we can introduce the notion of smooth transfer and state the
main theorem of the paper. We first fix Haar measures on $H$ and
$H'$, and fix a Haar measure on $H'_y$ for each $y$ in $S'_\rs$ or
$\fs'_\rs(F)$. We may and do assume that for $y_1,y_2$ in $S'_\rs$
or $\fs'_\rs(F)$ that lie in the same $H'$-orbit, the Haar measures
on $H'_{y_1}$ and $H'_{y_2}$ are compatible in the obvious sense.
For any $x$ in $S_\rs$ or $\fs_\rs(F)$, choose any $y$ in $S'_\rs$
or $\fs'_\rs(F)$ respectively so that $x\leftrightarrow y$. Then
$H_x\simeq H'_y$. We choose the Haar measure on $H_x$ compatible
with that on $H'_y$.

\begin{defn}
For $x\in S_\rs$ (resp. $x\in\fs_\rs(F)$), and $f\in\CC_c^\infty(S)$
(resp. $f\in\CC_c^\infty(\fs(F))$), we define the orbital integral
of $f$ at $x$ to be $$O(x,f)=\int_{H_x\bs H}f(h^{-1}xh)\ \d h,$$
which is convergent since any semisimple orbit is closed.
\end{defn}

\begin{defn}\label{defn. transfer}
For $f\in\CC_c^\infty(S)$ (resp. $\CC_c^\infty(\fs(F))$), we say
that $f'\in\CC_c^\infty(S')$ (resp. $\CC_c^\infty(\fs'(F))$) is a
smooth transfer of $f$ if for each $y\in S'_\rs$ (resp.
$y\in\fs'_\rs(F)$)
\begin{equation}\label{equation. condition of orbital integrals}
\begin{array}{lll}O(y,f')=\left\{\begin{array}{ll}
\begin{aligned}O(x,f),&\quad \textrm{if there exists } x\in S_\rs
\textrm{ (resp. $x\in\fs_\rs(F)$) such that }x\leftrightarrow y,\\
0,&\quad \textrm{otherwise}.
\end{aligned}
\end{array}\right.
\end{array}\end{equation}
Sometimes we will write transfer instead of smooth transfer for
short. If $f'$ is a transfer of $f$, we write $f\leftrightarrow f'$
for simplicity.
\end{defn}

\begin{remark}\label{rem. converse smooth transfer}
Conversely, for $f'\in\CC_c^\infty(S')$ (resp.
$f'\in\CC_c^\infty(\fs'(F))$) satisfying the following condition
\begin{equation}\label{equation. condition}
O(y,f')=0\ \textrm{if there does not exist $x$ in $S_\rs$ (resp.
$\fs_\rs(F)$) such that $x\leftrightarrow y$},
\end{equation} we say that $f\in\CC_c^\infty(S)$
(resp. $f\in\CC_c^\infty(\fs(F))$) is a smooth transfer of $f'$ if
for each $x\in S_\rs$ (resp. $x\in\fs_\rs(F)$)
$$O(x,f)=O(y,f'),$$ where $y$ in $S'_\rs$ (resp. $\fs'_\rs(F)$) is
such that $x\leftrightarrow y$.
\end{remark}

\begin{remark}\label{rem. transfer on descendant}
For semisimple $x\in S$ (resp. $x\in\fs(F)$) and semisimple $y\in
S'$ (resp. $y\in\fs'(F)$) such that $x\leftrightarrow y$, we can
also define smooth transfer from $\CC_c^\infty(\fs_x(F))$ to
$\CC_c^\infty(\fs'_y(F))$ determined by the orbital integrals with
respect to the action of $H_x$ on $\fs_x(F)$ and the action of
$H'_y$ on $\fs'_y(F)$. According to Remark \ref{rem. matching
descendants}, there are two types of smooth transfer to consider.
The first type is associated to $(\FH_{m'},\fs_{m'})$ and
$(\FH'_{m'd},\fs'_{m'd})$ with $m'\leq m$. The second type is
associated to $(\GL_r(D)_A,\fg\fl_r(D)_A)$ and its inner form
$(\GL_{rd,B},\fg\fl_{rd,B})$. In this case, the orbital integrals
are with respect to the adjoint action and the existence of smooth
transfer is known (cf. \cite{wa97}).
\end{remark}

Our main theorem on the smooth transfer is the following.

\begin{thm}\label{thm. main1}
For each $f\in\CC_c^\infty(S)$, there exists $f'\in\CC_c^\infty(S')$
that is a smooth transfer of $f$.
\end{thm}

Showing the existence of smooth transfer essentially is a local
issue. Via the Luna Slice Theorem and descent of orbital integrals,
we can reduce to proving the existence of smooth transfer between
the descendants $(\FH_x,\fs_x(F))$ and $(\FH'_y,\fs'_y(F))$ for each
semisimple $x\in S$ and $y\in S'$ such that $x\leftrightarrow y$.
According to Remark \ref{rem. transfer on descendant}, we reduce to
proving the following Lie algebra version of smooth transfer. We
refer the reader to \cite[\S3]{zh1} or \cite[\S5.3]{zh} for more
details of such reduction steps. The arguments there can be applied
for our situation without modification.

\begin{thm}\label{thm. main2}
For each $f\in\CC_c^\infty(\fs(F))$, there exists
$f'\in\CC_c^\infty(\fs'(F))$ that is a smooth transfer of $f$.
\end{thm}

To prove Theorem \ref{thm. main2}, the following theorem, which
roughly says that the Fourier transform commutes with smooth
transfer, is the key input.

\begin{thm}\label{thm. fourier}
There exists a nonzero constant $c\in\BC$ satisfying that: if
$f'\in\CC_c^\infty(\fs'(F))$ is a transfer of
$f\in\CC_c^\infty(\fs(F))$, then $\wh{f'}$ is a transfer of
$c\wh{f}$.
\end{thm}

\begin{remark}
We now briefly recall why Theorem \ref{thm. fourier} implies Theorem
\ref{thm. main2}. We use induction and assume that Theorem \ref{thm.
main2} holds for functions in $\CC_c^\infty(\fs_{m'}(F))$ for every
$m'<m$. Via the Luna Slice Theorem and descent of orbital integrals
again, we can reduce to proving the existence of smooth transfers on
the descendants for each semisimple $X\in\fs_\ss(F)$ and
$Y\in\fs'_\ss(F)$ such that $X\leftrightarrow Y$. If $X$ is nonzero,
the factor $(\FH_{m'},\fs_{m'})$ in the descendant $(\FH_X,\fs_X)$
satisfies that $m'<m$. Thus, by Remark \ref{rem. transfer on
descendant} and the inductive hypothesis, smooth transfers exist for
functions whose supports are contained in a neighborhood of $X$.
Moreover, this shows the existence of smooth transfer for
$f\in\CC_c^\infty(\fs(F)-\CN)$, where $\CN$ is the null-cone of
$\fs(F)$. We have explained that the symmetric pair $(\FG,\FH)$ is
of negative defect, which implies that $(\FG,\FH)$ is special (cf.
\cite[Proposition 7.3.7]{ag1}). The speciality means the following
statement. If $T$ is an $H$-invariant distribution on $\fs(F)$ such
that $\Supp(T)\subset\CN$ and $\Supp(\wh{T})\subset\CN$, then $T$
must be zero. This fact has the following direct consequence. Let
$\CC_0=\bigcap\limits_{T}\ker(T)$ where $T$ runs over all
$H$-invariant distributions on $\fs(F)$. Then each
$f\in\CC_c^\infty(\fs(F))$ can be written as $f=f_0+f_1+\wh{f_2}$
with $f_0\in\CC_0$ and $f_i\in\CC_c^\infty(\fs(F)-\CN)$ for $i=1,2$.
Therefore, it remains to prove the existence of the smooth transfer
for $\wh{f}$ with $f$ belonging to the space
$\CC_c^\infty(\fs(F)-\CN)$, which is exactly what Theorem \ref{thm.
fourier} shows.
\end{remark}

\begin{remark}
To prove the existence of smooth transfer in the converse direction,
in the sense of Remark \ref{rem. converse smooth transfer}, it
suffices to prove that each $f\in\CC_c^\infty(\fs'(F))$ satisfying
condition (\ref{equation. condition}) can be written as
$f=f_0+f_1+\wh{f_2}$. Here $f_0$ is in $\CC'_0$ which is defined
similarly as $\CC_0$ for $\fs$, and
$f_i\in\CC_c^\infty(\fs'(F)-\CN')$, for $i=1,2$, is also required to
satisfy condition (\ref{equation. condition}) (here $\CN'$ is the
null-cone of $\fs'(F)$). However we do not know how to prove such a
decomposition.
\end{remark}

\section{Local orbital integrals}\label{section. local orbital
integrals}

Let $F$ be a $p$-adic field as before. This section is devoted to
proving Theorem \ref{thm. fourier}. We employ several techniques
used by Waldspurger on endoscopic smooth transfer (cf. \cite{wa95}
and \cite{wa97}). These techniques also involve some more classical
results of Harish-Chandra on harmonic analysis for $p$-adic groups
(cf. \cite{hc1} and \cite{hc}). We also establish various analogous
results for the $p$-adic symmetric spaces considered here. A large
part of this section can be viewed as a generalization of the
results in \cite{zh}.

\subsection{Preparations}
\paragraph{Inequalities}
Fix a nonzero $X_0$ in the null-cone $\CN$ of $\fs(F)$. Let
$(X_0,{\bf d},Y_0)$ be an $\fs\fl_2$-triple with ${\bf d}\in\fh(F)$
and $Y_0\in\CN$. We set ${\bf r}=\dim_F\fs_{Y_0}$ and ${\bf
m}=\frac{1}{2}\RTr\left(\ad(-{\bf d})|_{\fs_{Y_0}}\right)$. The
inequalities below are used to bound the orbital integrals for
elements of $\CC_c^\infty(\fs(F))$ along a Cartan subspace of $\fs$.

\begin{prop}\label{prop. inequality for nilpotent}
We have the relations
\begin{enumerate}
\item ${\bf r}\geq n$,
\item ${\bf r}+{\bf m}> n^2+\frac{n}{2}$.
\end{enumerate}
\end{prop}

\begin{proof}
Let $L$ be an extension field of $F$ with degree $d$ contained in
$D$. Then
$$(\FG\times_FL,\FH\times_FL )\simeq(\FG'\times_FL,\FH'\times_FL)=:(\FG'',\FH'').$$
Denote by $\fs''$ the ``Lie algebra'' associated to $\FG''/\FH''$.
We can, in a canonical way, view $X_0$ and $Y_0$ as elements of
$\fs''(L)$ and also ${\bf d}$ as an element of $\fh''(L)$
($\fh'':=\Lie(\FH'')$). Let ${\bf r}'=\dim_L\fs''_{Y_0}$ and ${\bf
m}'=\frac{1}{2}\RTr\left(\ad(-{\bf d})|_{\fs''_{Y_0}}\right)$. Since
$\fs''_{Y_0}\simeq\fs_{Y_0}\otimes_FL$, it is not hard to see that
${\bf r}={\bf r}'$ and ${\bf m}={\bf m}'$. Therefore the required
inequalities follow immediately by \cite[Proposition 4.4]{zh}.
\end{proof}

\paragraph{Representability}
With the aid of Proposition \ref{prop. inequality for nilpotent}, we
can generalize all the results in \cite[\S5, \S6, \S7]{zh} when
$d=1$ to the more general case at hand. We will only state the
results and omit the proofs since they are obtained as almost
verbatim reproductions of those in \cite{zh}.

Let $X\in\fs_\rs(F)$ lie in a Cartan subspace $\fc$ of $\fs$. Then
the centralizer $\FT$ of $\fc$ in $\FH$ equals $\FH_X$. Thus $T$ is
a torus by Proposition \ref{prop. descendant 1 lie}. Write
$\ft=\Lie(\FT)$. We define the normalizing factor $|D^\fs(X)|_F$ to
be
$$|\det(\ad(X);\fh/\ft\oplus\fs/\fc)|_F^{\frac{1}{2}},$$ which is
also equal to $|\det(\ad(X);\fg/\fg_X)|_F^{\frac{1}{2}}$. We
consider the normalized orbital integral:
$$I(X,f)=|D^\fs(X)|_F^{\frac{1}{2}}O(X,f),\ \textrm{ for
}f\in\CC_c^\infty(\fs(F)),$$ which is a distribution on $\fs(F)$. We
also consider its Fourier transform:
$$\wh{I}(X,f):=I(X,\wh{f}),\quad\textrm{ for }f\in\CC_c^\infty(\fs(F)).$$

If $X\in\fs_\rs(F)$ and $Y\in\fs'_\rs(F)$ are such that
$X\leftrightarrow Y$, viewed as elements of $\M_{2m\times 2m}(D)$
and $\M_{2n\times 2n}(F)$ respectively, they have the same
characteristic polynomial (cf. Proposition \ref{prop. match
orbits}). Since the normalizing factor is determined by
characteristic polynomial, we see that
$$|D^\fs(X)|_F=|D^{\fs'}(Y)|_F.$$ Hence it
does not matter if we consider the smooth transfer with respect to
the normalized orbital integrals.

The following theorem is a generalization of \cite[Theorem 6.1]{zh}.
Its proof can be copied word for word from that of \cite[Theorem
6.1]{zh}. The ingredients of its proof are the analogues of
parabolic induction and Howe's finiteness theorem for our symmetric
spaces, together with bounds for normalized orbital integrals along
Cartan subspaces of $\fs(F)$.

\begin{thm}\label{thm. representability}
For each $X\in\fs_\rs(F)$, there exists a locally constant
$H$-invariant function $\wh{i}_X$ defined on $\fs_\rs(F)$ which is
locally integrable on $\fs(F)$, such that for any
$f\in\CC_c^\infty(\fs(F))$ we have
$$\wh{I}(X,f)=\int_{\fs_\rs(F)}\wh{i}_X(Y)f(Y)
|D^\fs(Y)|_F^{-1/2}\ \d Y.$$
\end{thm}

We will need a proposition that shows up in the course of the proof.
Recall that an element $X\in\fs_\rs(F)$ is called elliptic if its
centralizer $\FH_X$ is an elliptic torus. Denote by $\fs_\el(F)$ the
set of elliptic elements in $\fs_\rs(F)$. Suppose that
$X\in\fs_\rs(F)$ is of the form
$\begin{pmatrix}0&{\bf1}_m\\A&0\end{pmatrix}$. Also suppose that $X$
is not elliptic. Then $A\in\GL_m(D)$ is not elliptic in the usual
sense. Then there exists a proper Levi subgroup $\FM_0$ of
$\GL_m(D)$ such that $A\in M_0:=\FM_0(F)$. Set $\fm_0:=\Lie(\FM_0)$.
Identify $\fs^+$ (resp. $\fs^-$) with $\fg\fl_m(D)$, and let
$\fr^+\subset\fs^+$ (resp. $\fr^-\subset\fs^-$) be the subspace that
corresponds to $\fm_0$ under this identification. Finally, set
$\fr=\fr^+\oplus\fr^-$. Then $X$ lies in $\fr(F)$ and is regular
semisimple with respect to the adjoint action of $M=M_0\times M_0$
on $\fr(F)$. Choosing a Haar measure on $M$, we also consider the
orbital integral with respect to the action of $M$ on $\fr(F)$. Note
that $\ft$ is contained in $\fm$ and $\fc$ is contained in $\fr$.
The normalizing factor $|D^\fr(X)|_F$ is defined to be
$$|\det\left(\ad(X);\fm/\ft\oplus\fr/\fc\right)|_F^{\frac{1}{2}}.$$
The normalized orbital integral $I^\fr(X,f')$, for
$f'\in\CC_c^\infty(\fr(F))$, is defined to be
$$|D^\fr(X)|_F^{\frac{1}{2}}\int_{H_X\bs M}f'(m^{-1}Xm)\ \d
m,$$ which is convergent since $X$ is semisimple with respect to the
action of $M$. $I^\fr(X,\cdot)$ is a distribution on $\fr(F)$. We
also consider its Fourier transform
$$\wh{I}^\fr(X,f'):=I^\fr(X,\wh{f'}).$$ Then, with suitable choices
of Haar measures, there is a relation between the orbital integrals
$I(X,\cdot)$ and $I^\fr(X,\cdot)$, the so-called parabolic descent
of orbital integrals,
$$I(X,f)=I^\fr(X,f^{(\fr)}),\ \textrm{ for all
}f\in\CC_c^\infty(\fs(F)).$$ Here $f^{(\fr)}\in\CC_c^\infty(\fr(F))$
is a sort of ``constant term'' of $f$. We refer the reader to
\cite[\S6.1]{zh} for the precise definition. The exact same formula
for $f^{(\fr)}$ as there still works in our situation.

Applying Theorem \ref{thm. representability} to lower rank
situations, we see that there exists a locally constant
$M$-invariant function $\wh{i}^\fr_X$ defined on $\fr_\rs(F)$ which
is locally integrable on $\fr(F)$, such that for any
$f'\in\CC_c^\infty(\fr(F))$ we have
$$\wh{I}^\fr(X,f')=\int_{\fr_\rs(F)}\wh{i}_X^\fr(Y) f'(Y)
|D^\fr(Y)|_F^{-\frac{1}{2}}\ \d Y.$$ Not surprisingly, there is a
relation between $\wh{i}_X$ and $\wh{i}^\fr_X$. The following
formula for $\wh{i}_X$ in terms of $\wh{i}^\fr_X$ will be needed.

\begin{prop}\label{prop. i(X,Y) parabolic} Keep the notations and
assumptions above. We have
$$\wh{i}_X(Y)=\sum_{Y'}\wh{i}_X^\fr(Y'),\quad
Y\in\fs_\rs(F),$$ where $Y'$ runs over a set of representatives for
the finitely many $M$-conjugacy classes of elements in $\fr(F)$
which are $H$-conjugate to $Y$. In particular, if there is no
element in $\fr(F)$ which is $H$-conjugate to $Y$, we have
$$\wh{i}_X(Y)=0.$$
\end{prop}

\paragraph{Limit formula} We also write $\wh{i}(X,Y)$ for $\wh{i}_X(Y)$.
There is a limit formula for $\wh{i}(X,Y)$ shown in
\cite[Proposition 7.1]{zh}, which takes care of a situation where
$d=1$ (and where there is an additional quadratic character present
to deal with the more general, twisted, periods considered there). A
similar formula still holds for the case at hand and will be stated
below. Notice that changing the Haar measures on $H$ and $H_X$
multiplies $\wh{i}(X,Y)$ by a nonzero scalar. We do not specify the
Haar measures, and instead refer the reader to \cite{zh} for more
details. Results that follow this limit formula (Proposition
\ref{prop. i(X,Y)} below) do not depend on the choices of the
measures.

Let $\fc$ be a Cartan subspace of $\fs$,  $\FT$ the centralizer of
$\fc$ in $\FH$, and $\ft$ the Lie algebra of $\FT$. For
$X,Y\in\fc_\reg(F)$, we define a bilinear form $q_{X,Y}$ on
$\fh(F)/\ft(F)$ by
$$q_{X,Y}(Z,Z')=\pair{[Z,X],[Y,Z']},$$ where the pairing
$\pair{\cdot,\cdot}$ is the one introduced before. One can check
that $q_{X,Y}$ is nondegenerate and symmetric. One can also verify
that $q_{X,Y}=q_{Y,X}$. We will write
$\gamma_\psi(X,Y)=\gamma_\psi(q_{X,Y})$ for simplicity.

\begin{prop}\label{prop. i(X,Y)}
Let $X\in\fs_\rs(F)$ and $Y\in\fc_\reg(F)$. Then there exists
$N\in\BN$ such that if $\mu\in F^\times$ satisfies $v_F(\mu)<-N$, we
have the equality
$$\wh{i}(\mu X,Y)=\sum_{h\in T\bs H,\ h\cdot X\in\fc}\gamma_\psi
\left(\mu h\cdot X,Y\right)\psi\left(\pair{\mu h\cdot X,Y}\right).$$
\end{prop}

\begin{proof}
One can make an obvious modification of the proof of
\cite[Proposition 7.1]{zh} to apply it here.
\end{proof}

\paragraph{Construction of test functions} For $X,Y\in\fc_\reg(F)$,
there is a formula for $\gamma_\psi(X,Y)$, which is exhibited in
\cite[Proposition 7.3]{zh} when $d=1$. The formula for general $d$
has the same form. We will not state it here, since it involves much
more notation. The following lemma is used to construct certain test
functions required in Proposition \ref{prop. local prop} below.

\begin{prop}\label{prop. compare lemma}
Let $\fc$ be a Cartan subspace of $\fs$. Fix a Cartan subspace
$\fc'$ of $\fs'$ such that $\fc\leftrightarrow\fc'$. Then for any
$X,X'\in\fc_\reg(F)$ we have the equality
$$\gamma_\psi\left(X,X'\right)=\gamma_\psi(\fh(F))
\gamma_\psi(\fh'(F))^{-1}\gamma_\psi(\varphi_\fc(X),\varphi_\fc(X')).$$
Here $\varphi_\fc$ is an isomorphism from $\fc$ to $\fc'$ as in
Remark \ref{rem. matching cartan}.
\end{prop}

The following proposition is an analogue of \cite[Proposition
7.6]{zh}, and its proof involves Propositions \ref{prop. i(X,Y)} and
\ref{prop. compare lemma}. It plays an important role in proving the
existence smooth transfer using the global method that we are
following here.

\begin{prop}\label{prop. local prop}
Let $X_0\in\fc_\reg(F)$ and $Y_0\in\fc'_\reg(F)$ be such that
$X_0\leftrightarrow Y_0$. Then there exist functions
$f\in\CC_c^\infty(\fs(F))$ and $f'\in\CC_c^\infty(\fs'(F))$
satisfying the following conditions.
\begin{enumerate}
\item If $X\in\Supp(f)$, $X$ is $H$-conjugate to an element in $\fc_\reg(F)$.
If $Y\in\Supp(f')$, there exists $X'\in\fc_\reg(F)$ such that
$X'\leftrightarrow Y$.
\item $f'$ is a transfer of $f$.
\item There is an equality
$$\wh{I}(X_0,f)=c\wh{I}(Y_0,f')\neq0,$$
where $c=\gamma_\psi(\fh(F))\gamma_\psi(\fh'(F))^{-1}$.
\end{enumerate}
\end{prop}

\begin{proof}
The same proof as that of \cite[Proposition 7.6]{zh} applies.
\end{proof}

\subsection{Proof of Theorem \ref{thm. fourier}}

In this subsection, we fix two $\CC_c^\infty$-functions
$f'\in\CC_c^\infty(\fs'(F))$ and $f\in\CC_c^\infty(\fs(F))$ such
that $f\leftrightarrow f'$.  The proof of Theorem \ref{thm. fourier}
can be divided into two parts:
\begin{enumerate}
\item the first part is to prove that $\wh{I}(Y,f')=0$ for any
$Y\in\fs'_\rs(F)$ such that there exists no element in $\fs_\rs(F)$
matching $Y$;
\item the second part is to search for a nonzero constant $c\in\BC$,
independent of $f$ and $f'$, such that $$\wh{I}(Y,f')=c\wh{I}(X,f)$$
for any $X\in\fs_\rs(F),Y\in\fs'_\rs(F)$ such that $X\leftrightarrow
Y$.
\end{enumerate}

\paragraph{First part of the proof}
Now we fix a $Y_0\in\fs'_\rs(F)$ such that there exists no element
in $\fs_\rs(F)$ matching $Y_0$. Suppose that $Y_0$ belongs to a
Cartan subspace $\fc'_0$ of $\fs'$. By Theorem \ref{thm.
representability} (in the case where $d=1$) and the Weyl integration
formula, we have
\begin{equation}\label{equation. weyl}\begin{aligned}
\wh{I}(Y_0,f')&=\int_{\fs'_\rs(F)}\wh{i}_{Y_0}(Z)f'(Z)
|D^{\fs'}(Z)|_F^{-\frac{1}{2}}\ \d Z\\
&=\sum_{\fc'} \frac{1}{w_{\fc'}}\int_{\fc'_\reg(F)}\wh{i}_{Y_0}(Z)
I(Z,f')\ \d Z,
\end{aligned}\end{equation}
where $\fc'$ runs over a set of representatives for the finitely
many $H'$-conjugacy classes of Cartan subspaces in $\fs'$ and
$w_{\fc'}$ is the cardinality of the relative Weyl group associated
to $\fc'$. For the Weyl integration formula in the setting of
symmetric spaces, we refer the reader to \cite[page 106]{rr}.

We denote by $\sC^D$ the set of Cartan subspaces $\fc'$ of $\fs'$
such that there exists a Cartan subspace $\fc$ of $\fs$ with
$\fc\leftrightarrow\fc'$. By the condition on $Y_0$, we see that
$\fc'_0\notin\sC^D$.

For any $\fc'\notin\sC^D$, we automatically have $I(Z,f')=0$ for
each $Z\in\fc'_\reg(F)$ by the condition on $f'$. If $\fc'\in\sC^D$,
we claim that $\wh{i}_{Y_0}(Z)=0$ for any $Z\in\fc'_\reg(F)$. We can
assume that $Y_0$ is of the form $\begin{pmatrix}0&{\bf 1}_n\\ A&0
\end{pmatrix}$ with $A\in\GL_n(F)_\rs$. By the condition on $Y_0$,
there exists an irreducible factor (over $F$) of the characteristic
polynomial of $A$ with degree $r$ such that $d\nmid r$. Then there
exists a subspace $\fr$ of $\fs$ of the form
$\left(\fg\fl_{r}\oplus\fg\fl_{n-r}\right)\bigoplus
\left(\fg\fl_{r}\oplus\fg\fl_{n-r}\right)$ such that $Y_0\in\fr(F)$
(see Proposition \ref{prop. i(X,Y) parabolic} for the notation).
Since $\fc'\in\sC^D$, there exists no element in $\fr(F)$ which is
$H'$-conjugate to any $Z\in\fc'_\reg(F)$. Thus the claim follows
from Proposition \ref{prop. i(X,Y) parabolic}. Therefore, in any
case, we have showed that the terms appearing in the sum of
(\ref{equation. weyl}) are zero, thus obtaining that
$\wh{I}(Y_0,f')=0$.

\paragraph{Second part of the proof}
The arguments in this part are almost the same as those in
\cite[\S8]{zh}. We shall explain them briefly.

Now, we fix $f\in\CC_c^\infty(\fs(F))$, and
$f'\in\CC_c^\infty(\fs'(F))$ which is a transfer of $f$, and fix
$X_0\in\fs_\rs(F),Y_0\in\fs'_\rs(F)$ such that $X_0\leftrightarrow
Y_0$. Next, we choose some global data as follows.

\s{$\bullet$ \emph{Fields}}. We choose a number field $k$ and a
central division algebra $\BD$ over $k$ so that:
\begin{enumerate}
\item $k$ is totally imaginary;
\item there exists a finite place $w$ of $k$ such that $k_w\simeq F$
and $\BD(k_w)\simeq D$;
\item there exists another finite place $u$ of $k$ such that $\BD$
does not split over $k_u$. By conditions 1 and 2, such a finite
place $u$ exists.
\end{enumerate}
Such a number field $k$ and a division algebra $\BD$ do exist. See
\cite[Proposition in \S 11.1]{wa97}. From now on, we identify $k_w$
with $F$. We denote by $\CO_k$ the ring of integers of $k$, and by
$\BA$ the ring of adeles of $k$. We fix a maximal order $\CO_\BD$ of
$\BD$ containing $\CO_k$. We fix a continuous character on $\BA/k$
whose local component at $w$ is $\psi$, and henceforth use the
letter $\psi$ to denote this new (global) character.

\s{$\bullet$ \emph{Groups}}. We define a global symmetric pair
$(\BG,\BH)$ over $k$ with respect to $\BD$, so that the base change
of $(\BG,\BH)$ to $k_w$ is isomorphic to $(\FG,\FH)$. Thus if the
index of $\BD$ is $d'$, let
$(\BG,\BH)=(\GL_{2m'}(\BD),\GL_{m'}(\BD)\times\GL_{m'}(\BD))$ where
$m'd'=n$. Define the symmetric pair
$(\BG',\BH')=(\GL_{2n},\GL_n\times\GL_n)$ over $k$ as usual. We now
use $\fh$ (resp. $\fh'$) to denote the Lie algebra of $\BH$ (resp.
$\BH'$), and $\fs$ (resp. $\fs'$) to denote the ``Lie algebra'' of
$\BG/\BH$ (resp. $\BG'/\BH'$). Hence $X_0\in\fs_\rs(k_w)$ and
$Y_0\in\fs'_\rs(k_w)$.

\s{$\bullet$ \emph{Places}}. Denote by $V$ (resp. $V_\infty,\ V_\f$)
the set of all (resp. archimedean, non-archimedean) places of $k$.
Fix two $\CO_k$-lattices:
$\FL=\fg\fl_{m'}(\CO_\BD)\oplus\fg\fl_{m'}(\CO_\BD)$ in $\fs(k)$ and
$\FL'=\fg\fl_n(\CO_{k})\oplus\fg\fl_n(\CO_k)$ in $\fs'(k)$. For each
$v\in V_\f$, set $\FL_v=\FL\otimes_{\CO_k}\CO_{k,v}$ and
$\FL'_v=\FL'\otimes_{\CO_k}\CO_{k,v}$. We fix a finite set $S\subset
V$ such that:
\begin{enumerate}
\item $S$ contains $u,w$ and $V_\infty$;
\item for each $v\in V-S$, everything is unramified, i.e. $\BG$ and
$\BG'$ are unramified over $k_v$, and $\FL_v$ and $\FL'_v$ are
self-dual with respect to $\psi_v$ and $\pair{\ ,\ }$.
\end{enumerate}
We denote by $S'$ the subset $S-V_\infty-\{w\}$ of $S$.

\s{$\bullet$ \emph{Orbits}}. For each $v\in V_\f$, we choose an open
compact subset $\Omega_v$ of $\fs(k_v)$ such that:
\begin{enumerate}
\item if $v=w$, we require that: $X_0\in\Omega_w$ and
$\Omega_w\subset\fs_\rs(k_w)$, $\wh{I}(\cdot,f)$ is constant on
$\Omega_w$, and $\wh{I}(\cdot,f')$ is constant and hence equal to
$\wh{I}(Y_0,f')$ on the set of $Y\in\fs'_\rs(k_w)$ which matches an
element $X$ in $\Omega_w$;
\item if $v=u$, we require $\Omega_u\subset\fs_\el(k_u)$;
\item if $v\in S$ but $v\neq w,u$, choose $\Omega_v$
to be any open compact subset of $\fs(k_v)$;
\item if $v\in V_\f-S$, let $\Omega_v=\FL_v$.
\end{enumerate}
Recall that a semisimple regular element $X\in\fs(k)$ is called
elliptic if its centralizer $\BH_X$ is an elliptic torus. Denote by
$\fs_\el(k)$ (resp. $\fs'_\el(k)$) the set of elliptic regular
semisimple elements in $\fs(k)$ (resp. $\fs'(k)$). Then by the
strong approximation theorem, there exists
$X^0\in\fs(k)\subset\fs(\BA)$ such that for each $v\in V_\f$ we have
$X^0\in\Omega_v$. Furthermore, by the condition (2) on the
$\Omega_v$'s, $X^0\in\fs_\el(k)$. Take an element
$Y^0\in\fs'_\el(k)$ such that $X^0\leftrightarrow Y^0$.

\s{$\bullet$ \emph{Functions}}. For each $v\in V$, we choose
Bruhat-Schwartz functions $\phi_v\in\CS(\fs(k_v))$ and
$\phi'_v\in\CS(\fs'(k_v))$ in the following way:
\begin{enumerate}
\item if $v=w$, let $\phi_v=f$ and $\phi'_v=f'$;
\item if $v\in S'$, by Proposition \ref{prop. local prop}, we can
require that:
\begin{itemize}
\item if $X_v\in\Supp(\phi_v)$, there exists $X'_v\in\fc_{X^0}(k_v)$
such that $X_v$ and $X'_v$ are $\BH(k_v)$-conjugate, where
$\fc_{X^0}$ is the Cartan subspace of $\fs$ containing $X^0$;
\item if $Y_v\in\Supp(\phi'_v)$, there exists $X_v\in\fc_{X^0}(k_v)$
such that $X_v\leftrightarrow Y_v$;
\item $\phi'_v$ is a transfer of $\phi_v$;
\item $\wh{I}(X^0,\phi_v)=c_v\wh{I}(Y^0,\phi'_v)$, where
$c_v=\gamma_\psi(\fh(k_v))\gamma_\psi(\fh'(k_v))^{-1}$;
\end{itemize}
\item
for $v\in V-S$, since we required $\BG$ to be unramified over $k_v$,
that is to say, $\BD$ to be split over $k_v$, we can make suitable
identifications  $\BG(k_v)=\BG'(k_v)$, $\FL_v=\FL'_v$, and set
$\phi_v=\phi_v'={\bf1}_{\FL_v}$; moreover, since $\FL_v$ is
self-dual with respect to $\psi_v$ and $\pair{\ ,\ }$,
$\phi_v=\wh{\phi}_v$;
\item for $v_0\in V_\infty$, identifying $(\BH(k_{v_0}),\fs(k_{v_0}))$ with
$(\BH'(k_{v_0}),\fs'(k_{v_0}))$, we can choose
$\phi_{v_0}=\phi'_{v_0}\in\CS(\fs(k_{v_0}))$ such that:
\begin{itemize}
\item $\wh{I}(X^0,\phi_{v_0})=\wh{I}(Y^0,\phi'_{v_0})\neq0$;
\item if $X\in\fs(k)$ is $\BH(k_v)$-conjugate to an element in
the support of $\wh{\phi_v}$ at each place $v\in V$, then $X$ is
$\BH(k)$-conjugate to $X^0$;
\item if $Y\in\fs'(k)$ is $\BH'(k_v)$-conjugate to an element in
the support of $\wh{\phi'_v}$ at each place $v\in V$, then $Y$ is
$\BH'(k)$-conjugate to $Y^0$.
\end{itemize}
\end{enumerate}
The condition 4 can be satisfied, and was discussed in \cite[Lemma
in \S10.7]{wa97} in the endoscopic case. The key point is that we
have a morphism $\fs/\BH\ra\FA_k^\ell$ where
$\FA_k^\ell=\Spec\left(\CO(\fs)^\BH\right)$ is an affine space. Then
the discussion is the same as in \cite[Lemma in \S10.7]{wa97}.

Now we set $\phi\in\CS(\fs(\BA))$ and $\phi'\in\CS(\fs'(\BA))$ to
be:
$$\phi=\prod_{v\in V}\phi_v,\quad \phi'=\prod_{v\in V}\phi'_v.$$

\s{$\bullet$ \emph{The end of the proof}}. As shown in \cite[Theorem
8.2]{zh} the following integrals $I(\phi)$ and $I(\phi')$ are
absolutely convergent:
$$I(\phi)=\int_{\BH(k)\bs\BH(\BA)^1}\sum_{X\in\fs_\el(k)}\phi(X^h)\ \d h,\quad
I(\phi')=\int_{\BH'(k)\bs\BH'(\BA)^1}\sum_{Y\in\fs'_\el(k)}\phi'(Y^h)\
\d h,$$ where
$$\BH(\BA)^1=\bigcap_{\chi\in\Hom_k(\BH,\BG_m)}\ker|\chi|,
\quad \BH'(\BA)^1=\bigcap_{\chi\in\Hom_k(\BH',\BG_m)}\ker|\chi|.$$
Here $\abs{\chi}$ is the function on $\BH(\BA)$ or $\BH'(\BA)$
defined in the usual way. Actually \cite[Theorem 8.2]{zh} only
treats the case of $(\BG',\BH')$, but the arguments also work for
$(\BG,\BH)$. It is obvious that
$$I(\phi)=\sum_{X\in[\fs_\el(k)]}\tau(X)\prod_v I(X,\phi_v),$$
and $$I(\phi')=\sum_{Y\in[\fs'_\el(k)]}\tau(Y)\prod_v
I(Y,\phi'_v),$$ where $[\fs_\el(k)]$ denotes the set of
$\BH(k)$-orbits in $\fs_\el(k)$, $$\tau(X)=\vol(\BH_X(k)\bs
(\BH_X(\BA)\cap\BH(\BA)^1)),$$ and the definitions of
$[\fs'_\el(k)]$ and $\tau(Y)$ are similar. If $X\in\fs_\el(k)$ and
$Y\in\fs'_\el(k)$ are such that $X\leftrightarrow Y$, then
$\BH_X\simeq\BH'_Y$ (the justification is the same as in the local
field case). We choose Haar measures on $\BH_X(\BA)$ and
$\BH'_Y(\BA)$ so that they are compatible. In particular,
$$\tau(X)=\tau(Y).$$

According to the conditions on $\phi_u$ (resp. $\phi'_u$), we know
that if $X\in\fs(k)$ (resp. $Y\in\fs'(k)$) is such that
$X\in\Supp(\phi)^{\BH(\BA)}$  (resp.
$Y\in\Supp(\phi')^{\BH'(\BA)}$), then $X\in\fs_\el(k)$ (resp.
$Y\in\fs'_\el(k)$). Here we use $\Supp(\phi)^{\BH(\BA)}$ to denote
the union of all $\BH(\BA)$-orbits intersecting $\Supp(\phi)$. We
have a similar defined set $\Supp(\phi')^{\BH'(\BA)}$. By the
conditions on $\phi'_v$ at each place $v$, we know that, if
$Y\in\fs'_\el(k)$ is such that $I(Y,\phi'_v)\neq0$ for each $v\in
V$, then there exists $X_v\in\fs_\rs(k_v)$ such that
$X_v\leftrightarrow Y$ at each place $v\in V$ and hence there exists
$X\in\fs_\el(k)$ such that $X\leftrightarrow Y$. Therefore we have
$$I(\phi)=I(\phi'),$$ since $\phi_v$ is a transfer of $\phi'_v$
at each place $v\in V$ by the requirements we have imposed.

On the other hand, according to the conditions on $\wh{\phi_v}$ and
$\wh{\phi'_v}$, we know that if $X\in\fs(k)$ (resp. $Y\in\fs'(k)$)
is such that $X\in\Supp(\wh{\phi})^{\BH(\BA)}$ (resp.
$Y\in\Supp(\wh{\phi'})^{\BH'(\BA)}$), $X$ is $\BH(k)$-conjugate to
$X^0$ (resp. $Y$ is $\BH'(k)$-conjugate to $Y^0$).

By the Poisson summation formula, we have
$$\sum_{X\in\fs(k)}\phi(X^h)=\sum_{X\in\fs(k)}\wh{\phi}(X^h),\quad
\forall\ h\in\BH(\BA),$$ and
$$\sum_{Y\in\fs'(k)}\phi'(Y^h)=\sum_{Y\in\fs'(k)}\wh{\phi'}(Y^h),\quad
\forall\ h\in\BH'(\BA).$$ Actually, by the conditions on $\phi$ and
$\phi'$, we can replace $\fs(k)$ (resp. $\fs'(k)$) by $\fs_\el(k)$
(resp. $\fs'_\el(k)$) on both sides of the above two equations.
Thus, we have
$$I(\phi)=I(\wh{\phi}),\quad I(\phi')=I(\wh{\phi'}).$$ Hence we have
$$I(\wh{\phi})=I(\wh{\phi'}),$$
which amounts to saying,
$$\tau(X^0)\prod_{v\in V}\wh{I}(X^0,\phi_v)
=\tau(Y^0)\prod_{v\in V}\wh{I}(Y^0,\phi'_v).$$ For $v\in V-S$ or
$v\in V_\infty$, we have
$$\wh{I}(X^0,\phi_v)=\wh{I}(Y^0,\phi'_v)\neq0.$$
For $v\in S'$ we have
$$\wh{I}(X^0,\phi_v)=c_v\wh{I}(Y^0,\phi'_v)\neq0.$$ Therefore
$$c\wh{I}(X^0,f)=\wh{I}(Y^0,f'),$$ where
$$c=\prod_{v\in S'}c_v=\prod_{v\in S'}\gamma_\psi(\fh(k_v))
\gamma_\psi(\fh'(k_v))^{-1}.$$ Notice that if $v\in V_\infty$ or
$v\in V-S$,
$$\gamma_\psi(\fh(k_v))=\gamma_\psi(\fh'(k_v))=1.$$ Also notice that
$$\prod_{v\in V}\gamma_\psi(\fh(k_v))=\prod_{v\in
V}\gamma_\psi(\fh'(k_v))=1.$$ Therefore
$$c=\gamma_\psi(\fh(k_w))^{-1}\gamma_\psi(\fh'(k_w)).$$ Since
$$\wh{I}(X_0,f)=\wh{I}(X^0,f),
\quad \wh{I}(Y_0,f')=\wh{I}(Y^0,f'),$$ we complete the proof of the
theorem.

\paragraph{Acknowledgements} This work was supported by the National Key
Basic Research Program of China (No. 2013CB834202). The author would
like to thank Dipendra Prasad for his valuable comments, and Wen-Wei
Li for his long list of useful comments and suggestions. He also
thanks Dihua Jiang and Binyong Sun for helpful discussions. He
expresses gratitude to Ye Tian and Linsheng Yin for their constant
encouragement and support. The anonymous referee pointed out a gap
and numerous mathematical and grammatical inaccuracies, made a lot
of useful comments, and helped the author to greatly improve the
exposition. The author is grateful to him or her.

\s{\small Chong Zhang\\
School of Mathematical Sciences, Beijing Normal University,\\
Beijing 100875, P. R. China.\\
E-mail address: \texttt{zhangchong@bnu.edu.cn}}

\end{document}